\documentclass[12pt]{article}
\usepackage{amssymb,amsmath,amsthm, amsfonts}
\DeclareMathOperator{\esssup}{ess\,sup}
\DeclareMathOperator{\essinf}{ess\,inf}

\usepackage{graphicx}
\usepackage{epsfig}
\usepackage{tikz}
\theoremstyle{plain}
\newtheorem{theorem}{Theorem}[section]
\newtheorem{lemma}{Lemma}[section]

\theoremstyle{definition}
\newtheorem{definition}{Definition}[section]

\numberwithin{equation}{section}

\allowdisplaybreaks \numberwithin{equation} {section}

\begin{document}
\allowdisplaybreaks
\title{\large\bf Stabilization of arbitrary structures in a  three-dimensional doubly degenerate nutrient taxis system}
\author{
{\rm De-Ji-Xiang-Mao, Ai Huang and Yifu Wang}\\
{\it\small  School of Mathematics and Statistics,
 Beijing Institute of Technology,100081, Beijing, P.R. China}
}

\date{}
\maketitle
\begin{abstract}
The doubly degenerate nutrient taxis system
\begin{equation}\label {0.1}
\left\{
\begin{aligned}
&u_{t}=\nabla \cdot (uv\nabla u)-\chi \nabla \cdot (u^{\alpha}v\nabla v)+\ell uv,&x\in \Omega,\, t>0,\\
& v_{t}=\Delta v-uv,&x\in \Omega,\, t>0,\\
\end{aligned}
\right.
\end{equation}
is considered under zero-flux boundary conditions in a smoothly bounded domain $\Omega\subset\mathbb{R}^3$ where $\alpha>0,\chi>0$ and $\ell> 0$. By developing a novel class of functional inequalities to address the challenges posed by
   the doubly degenerate diffusion mechanism in \eqref{0.1}, it is shown that for $\alpha\in(\frac{3}{2},\frac{19}{12})$, the associated initial-boundary value problem admits a global continuous  weak solution for sufficiently regular initial data.
Furthermore, in an appropriate topological setting, this solution converges to an equilibrium $(u_\infty, 0)$ as $t\rightarrow \infty$. Notably, the limiting  profile $u_{\infty}$ is  non-homogeneous when the initial signal concentration $v_0$ is sufficiently small, provided the initial data $u_0$ is not identically constant.
\end{abstract}

\vskip1mm
{\small {Mathematical Subject Classifications}: 35K59, 35K65, 35B40, 92C17.

}

\vskip2mm

\section{Introduction}\label{Intro}
Significant research efforts have been devoted to characterizing the conditions under which parabolic systems can accurately model pattern formation in chemotaxis processes, ensuring consistency with experimental observations. A cornerstone of this research is the Keller-Segel model, in which the chemotactic-diffusion mechanism enables the description of collective microbial behavior driven by chemical gradients (\cite{KS,B,Ni}).

 Recent studies have demonstrated the crucial influence of chemotactic sensitivity function  on the dynamic prospects of chemotaxis system
 \begin{equation}\label{1.1}
\left\{
\begin{aligned}
&u_{t}= \nabla \cdot (D(u)\nabla u)-\nabla \cdot (S(u)\nabla v),&x\in \Omega,\, t>0,\\
& v_{t}=\Delta v -v+u,&x\in \Omega,\, t>0,\\
\end{aligned}
\right.
\end{equation}
posed on a smoothly bounded $n-$dimensional domain.
  Specifically, the ratio $\frac{S(u)}{D(u)}$ critically determines
whether solutions to the zero-flux initial-boundary value problem of \eqref{1.1}
are globally bounded or posses finite-time singularities. Indeed, when
 the diffusivity and sensitivity functions $D,S\in C^2(\mathbb R_+)$ satisfy $S(0)=0$ and the growth condition  $\frac{S(u)}{D(u)}\leq Cu^{\alpha}$ for all $u\geq1$ with $C>0$ and $\alpha<\frac{2}{n}$, the system  admits a global bounded classical solution (\cite{SKT, TaoW}). In contrast, when $\frac{S(u)}{D(u)}\geq Cu^{\alpha}$ for some $C>0$ and $\alpha>\frac{2}{n}$ at large values, the solution exhibits finite-time blow-up through concentration phenomena (\cite{TC1, TC2}).

 When chemotaxis directs motion  of bacteria (with density $u$) toward higher concentrations of nutrient (of
concentration $v$), the resulting dynamics are mathematically described by chemotaxis-consumption models of the form
\begin{equation}\label{1.2}
\left\{
\begin{aligned}
&u_{t}= \nabla \cdot (D(u)\nabla u)-\nabla \cdot (S(u)\nabla v),&x\in \Omega,\, t>0,\\
& v_{t}=\Delta v -vu,&x\in \Omega,\, t>0.\\
\end{aligned}
\right.
\end{equation}
In sharp contrast to the situation in \eqref{1.1}, due to the essentially dissipative character of   $v$-equation in \eqref{1.2},
 the latter system should exhibit some considerably stronger relaxation properties.  Notably, when considering the specific case with constant diffusivity $D(u)=1$  and linear sensitivity $S(u)=u$, subject to no-flux boundary conditions in bounded convex domains, the available literature reveals no evidence of non-trivial large-time dynamics for solutions of \eqref{1.2}.
 At the level of rigorous mathematical analysis, it is confirmed that for reasonably regular but arbitrarily large initial data, two-dimensional versions of (1.2) are known to possess global bounded classical  solutions (see \cite{Jinwang,TaoW2} for example); whereas when $n =3$, the certain global weak solutions  emanating from large initial data
 at least eventually become bounded and smooth, and particularly  approach the
unique relevant constant steady state in the large time limit (\cite{TaoW2,WinklerJDE2017}).

Recent experimental studies indicate that motility constraints in \it{Bacillus subtilis}
\rm  strains under nutrient-poor environments play a crucial role in the formation of diverse morphological aggregation patterns, including densely accumulating multiply branches structures (\cite{Fuji1, Fuji2, Matsushita}).
To gain a comprehensive understanding of the underlying mechanism, particularly the hypothesis that bacterial motility decreases under nutrient deprivation, a doubly degenerate nutrient taxis system of the form
\begin{equation}\label{1.3}
\left\{
\begin{aligned}
&u_{t}=\nabla \cdot (u^{m-1}v\nabla u)-\chi \nabla \cdot (S(u)v\nabla v)+f(u,v),&x\in \Omega,\, t>0,\\
& v_{t}=\Delta v-uv,&x\in \Omega,\, t>0,\\
\end{aligned}
\right.
\end{equation}
with $m=2, S(u)=u^2$, was  introduced  in \cite{Kawasaki, Leyva, Plaza}.
A remarkable feature of \eqref{1.3} is the doubly degenerate structure of the first equation, arising from two distinct mechanisms:
  (i) porous-medium-type degeneracy as $u$ approaches zero, and (ii) cross-degeneracy induced by the factor $ v$, which vanishes when $v$ tends to zero (as governed by the second equation).
  This peculiarity brings about noticeable mathematical challenges beyond classical porous medium-type diffusion theory.
  Hence, current understanding of even fundamental existence questions is still largely restricted to the global solvability in low-dimensional settings.

It is shown in  \cite{WCVPDE} that when $\chi=0$, $f(u,v)=0$ and $\Omega \subset\mathbb{R}^n$ is a bounded convex domain,
  the associated initial-boundary value problem for \eqref{1.3} possesses a global weak solution for the reasonably regular initial data. Moreover, within an appropriate topological setting,
  the solution can approach the non-homogeneous  steady-state $(u_\infty,0)$ in the large time limit.
  It is noticed that due to the  absence of taxis effects in \eqref{1.3}, the comparison principle
  allows for deriving local bounds for  $\|u(\cdot,t)\|_{L^\infty(\Omega)}$, and
   the boundedness of $\int_0^\infty\int_\Omega u^{q-1}v|\nabla u|^2$ for $ q\in(0,1) $ becomes the basic regularity property of \eqref{1.3}.  These estimates provide the essential foundation for establishing the global existence of  weak solution.
 However, when $\chi>0$,   the chemoattraction mechanisms in \eqref{1.3} may  exert considerably destabilizing effect,
 as evidenced by blow-up phenomena observed in analogous modeling contexts
of \eqref{1.1} (\cite{B,LW}).
 To the best of our knowledge,  the available analytical results for \eqref{1.3} with $\chi>0, f(u,v)=\ell uv$
  so far remain restricted to low-dimensional settings:

  i)
    In the \bf{one-dimensional} \rm case, when $m=2$ and $S(u)=u^2$, global existence of continuous weak solutions and their asymptotic stabilization toward non-homogeneous equilibria was established in \cite{Winkler3,LiWinkler1}.  These results were later extended to cases where either
    $2\leq m<3$ and $S(u)\leq Cu^{\alpha}$ for
  $\alpha\in[m-1,\frac{m}{2}+1]$ or $3\leq m<4$ (\cite{Wu}).

 ii) In \bf two-dimensional \rm counterparts,   global existence of weak solutions was first established  for $m=2$ and $S(u)=u^{\alpha}$ with  $1<\alpha<\frac{3}{2}$  (\cite{LiG}).  Subsequent work in \cite{WJDE}   extended this result to all  $\alpha<2$, and in particular, proved $L^{\infty}$-bounds for the solutions.  Notably, for the critical case
  $\alpha=2$, global existence of weak solutions was obtained under a smallness condition solely on the initial data $v_0$ in \cite{WNARWA}. This restriction was later  relaxed in \cite{LiYu}. Furthermore, the large-time behavior of solutions to
  \eqref{1.3}  has been studied in \cite{Wu1, Winkler5}  using techniques based on Harnack inequalities.

   In striking contrast to lower-dimensional cases where effective embedding theorems facilitate the analysis,  the study of the \bf{three-dimensional} \rm  version of \eqref{1.3} becomes significantly more subtle. To date, even fundamental questions regarding global solvability and boundedness, including the large-time behavior, remain largely unestablished in the literature.
 For instance,  the global solvability of \eqref{1.3}
    was established in \cite{LiG} when $\frac{7}{6}<\alpha<\frac{13}{9}$, which may be relaxed to $1<\alpha<\frac{3}{2}$
     (\cite{Wu}).
     Apart from that, the stabilizing effect of logistic-type source terms $f(u,v)=\rho u-\mu u^k$ on \eqref{1.3} is also investigated. For example, global existence of weak solutions was established
       in arbitrary dimensions $n\geq2$ under the
 condition  $k>\frac{n+2}{2}$ (\cite{XP}), and particularly the continuity of  weak solutions
was achieved when $n=k=2$ (\cite{LiWinkler2}).

The intention of the present work  is to develop an analytical approach that not only establishes the global solvability of system \eqref{1.3}, but also characterize the non-trivial long-time dynamics in three-dimensional domains.
 Specially, we shall consider  the initial-boundary value  problem
\begin{equation}\label{1.4}
\left\{
\begin{aligned}
&u_{t}=\nabla \cdot (uv\nabla u)-\chi \nabla \cdot (u^{\alpha}v\nabla v)+\ell uv,&x\in \Omega,\, t>0,\\
& v_{t}=\Delta v-uv,&x\in \Omega,\, t>0,\\
&(uv\nabla u-\chi u^{\alpha}v\nabla v)\cdot \nu=\nabla v\cdot \nu=0,&x\in \partial\Omega,\, t>0,\\
&u(x,0)=u_0(x), v(x,0)=v_0(x), &x\in \Omega,
\end{aligned}
\right.
\end{equation}
posed in a smoothly bounded domain $\Omega\subset \mathbb{R}^3$. Here, the parameters satisfy    $\chi>0$,  $\ell>0$  and $\alpha \in
(\frac{3}{2},\frac{19}{12})$.
The initial data $(u_0,v_0)$ are assumed to satisfy
\begin{equation}\label{1.5}
\left\{
\begin{aligned}
&u_0\in W^{1,\infty}(\Omega)\; \hbox{is  such  that}\; u_0\geq 0\; \hbox{and}\; u_0\not\equiv0, \quad \text{and} \\
&v_0\in W^{1,\infty}(\Omega)\; \hbox{is  such  that}\; v_0>0\,\, \text{in} \,\,\Omega\\
\end{aligned}
\right.
\end{equation}
as well as
\begin{align}\label{1.6}
    \|u_0\|_{L^{\infty}(\Omega)}+\|v_0\|_{L^{\infty}(\Omega)}+\|\nabla \ln v_0\|_{L^{\infty}(\Omega)}\leq K\quad \text{and}\quad \int_{\Omega}\ln{u_0}\geq -K
\end{align}
for some $K>0$.

In this  framework, we will firstly address the basic issue of the global existence and
boundedness of continuous weak solutions to \eqref{1.4}  when  the convexity of $\Omega$ is not necessary.

\begin{theorem}\label{Th1.1}
Let $\Omega \subset \mathbb R^3$ be a bounded domain with smooth boundary, and suppose that $\alpha \in \left( \frac{3}{2}, \frac{19}{12} \right)$ as well as
$\chi >0$ and $\ell>0$. Then for all $p>1$, one can find $C=C(p,K)>0$ with the property that whenever $u_0$ and $v_0$ fulfill
\eqref{1.5} and \eqref{1.6},
there exist
\begin{equation}\label{1.8}
\left\{
\begin{aligned}
&u\in C^{0}_{loc}(\overline{\Omega}\times [0,\infty))\cap L^{\infty}_{loc}(\overline{\Omega}\times [0,\infty)) \quad \text{and} \\
&v\in C^{0}_{loc}(\overline{\Omega}\times [0,\infty))\cap C^{2,1}_{loc}(\overline{\Omega}\times (0,\infty))\cap L^{\infty}_{loc}([0,\infty);W^{1,\infty}(\Omega))\\
\end{aligned}
\right.
\end{equation}
such that  $u\geq 0$ and $v>0$ a.e. in $\Omega\times(0,\infty)$, and that
$(u,v)$ forms a continuous global weak solution of \eqref{1.4} in the sense of Definition \ref{Definition 2.1} below. Moreover, we have
\begin{align}\label{1.9}
    \|u(\cdot,t)\|_{L^{p}(\Omega)}+\|v(\cdot,t)\|_{W^{1,\infty}(\Omega)}\leq C(p,K)\qquad for \;a.e.\; t>0.
\end{align}
\end{theorem}
By carefully tracking respective dependence on time, the estimates gained in the
proof of Theorem 1.1 already pave the way
for our subsequent asymptotic analysis. Indeed, by means of a duality-based argument, we can achieve
the following result on stability property enjoyed by function pairs $(u_0,0)$ for all suitably regular  $u_0$.  
\begin{theorem}\label{Th1.2}
Let $\Omega \subset \mathbb R^3$ be a bounded domain with smooth boundary, and suppose that $\alpha \in \left( \frac{3}{2}, \frac{19}{12} \right)$ as well as
$\chi >0$ and $\ell>0$. Then for each $\eta>0$, there exists $\delta_1=\delta_1(\eta, K)>0$ whenever $u_0$ and $v_0$ fulfill
\eqref{1.5} and \eqref{1.6}, as well as
$$\int_{\Omega}v_0\leq \delta_1,$$
the solution $(u,v)$ of \eqref{1.4} obtained in Theorem \ref{Th1.1} satisfies
\begin{align}\label{1.11-1}
    \|u(\cdot,t)-u_0\|_{(W^{1,\infty}(\Omega))^*}\leq \eta\qquad for \;all\; t>0.
\end{align}
\end{theorem}
Beyond that, we undertake a deeper qualitative analysis of the asymptotic properties of solutions.
Although Theorem \ref{Th1.2} guarantees that the solution component $u$ remains close to its initial data $u_0$ in the weak-$\ast$ topology of $(W^{1,\infty}(\Omega))^*$ for small $v_0$,  the fundamental question of whether the limiting profile  $u_\infty$  must be constant remains open.  We resolve this question by proving that when $u_0\not\equiv const$, sufficiently small initial concentrations of $v_0$  necessarily yield  nonconstant limiting profiles
 $u_\infty$.
\begin{theorem}\label{Th1.3}
Let $\Omega \subset \mathbb R^3$ be a bounded domain with smooth boundary, and suppose that $\alpha \in \left( \frac{3}{2}, \frac{19}{12} \right)$ as well as
$\chi >0$ and $\ell>0$. Then  there exists a nonnegative function $u_{\infty}\in \bigcap_{p\geq1}L^p(\Omega)$ such that as $t\to \infty$, the solution $(u,v)$ of \eqref{1.4} from Theorem \ref{Th1.1} satisfies
\begin{align}\label{1.10-1}
    u(\cdot,t)\rightharpoonup u_{\infty}\qquad in\;L^p(\Omega),\qquad v(\cdot,t)\rightarrow 0\qquad in\;W^{1,p}(\Omega)\; for \;all\; p\geq 1.
\end{align}
Moreover, for the given nonnegative function $u_0\not\equiv$ const., one can find
  $\delta_2=\delta_2(K, u_0)>0$ such that whenever $u_0$ and $v_0$ fulfill
\eqref{1.5} and \eqref{1.6}, as well as
$$\int_{\Omega}v_0\leq \delta_2,$$
the limit function satisfies
 $ u_{\infty}\not\equiv\mathrm{const}$.
\end{theorem}

\textbf{Main ideas.} The main challenge  in developing a mathematical theory for system
 \eqref{1.4} lies in effectively quantifying its dissipative mechanisms,  which
  are substantially weakened by the signal-dependent degeneracy in  the $u$-equation.
  Accordingly, existing literature has been limited to \eqref{1.4} posed on the bounded convex
 domain
 $\Omega \subset\mathbb{R}^n, n\leq 2$
  (\cite{Winkler3,WNARWA,WJDE,WCVPDE,LiWinkler1,LiG}). In non-convex three-dimensional domains, however, the analysis becomes considerably more intricate due to the loss of favorable embedding properties.

        Our approach is to improve the regularity of  \eqref{1.4} via the bootstrap argument. The most crucial step is to obtain the global boundedness of $\|u(\cdot,t)\|_{L^{p_0}(\Omega)}$ for some $p_0>\frac{3}{2}$. To this end,  we shall appropriately utilize the diffusion-induced contributions of the form $\int_{\Omega}u^{k}v|\nabla u|^2$ (for suitable $k$) to counteract the unfavorable terms typified by $\int_{\Omega}u^{\beta}v$.
    As the starting point of our analysis, we certify  that the integrals $\int_{\Omega}u^{1-\alpha}v|\nabla u|^2$ and $\int_{\Omega}u\frac{|\nabla v|^{2}}{v}$ are dissipated by tracking the
time evolution of $\int_{\Omega}u^{2-\alpha}$ and ${F}(t):=a\int_{\Omega} \frac{|\nabla v(\cdot,t)|^2}{v(\cdot,t)}-\int_{\Omega}\ln{u(\cdot,t)}$ (with some $a>0$) along trajectories of  \eqref{1.4}, as detailed in Lemma 2.3 and Lemma 2.5, respectively, which distinguishes our approach from that in \cite{LiG}. Moreover, differing from \cite{LiG,Wu}, our derivation of bounds for quantities such as
        \begin{align}\label{1.14}
        \int_0^{T}\int_\Omega u^{\frac{5}{3}}v
        \end{align}
         is to make use of the aforementioned dissipation terms
         and an interpolation inequality (Lemma \ref{lemma2.6}). Building upon \eqref{1.14}, we can estimate the
         unfavorable terms arising from the time evolution of the functional
         $$ \mathcal{ G}(t):=\int_{\Omega}u(\cdot,t)\ln{u}(\cdot,t)+b\int_{\Omega}\frac{|\nabla v(\cdot,t)|^4}{v^3(\cdot,t)}$$
        for some $b>0$, and thereby derive  the dissipated quantity of the form
        \begin{align}\label{1.15}
            \int_{\Omega}u^{p_*}v|\nabla u|^2  \quad \text{and}\quad  \int_{\Omega}u\frac{|\nabla v|^4}{v^3}
        \end{align}
         for all $p_*\in[1-\alpha, 0]$ (Lemma \ref{lemma3.3} and Lemma \ref{lemma3.4}). A further crucial ingredient in our analysis is the novel observation:
          Assume that $\Omega\subset \mathbb{R}^3$ is a  smoothly bounded domain, $L>0, k\in(-1,-\frac{1}{3})$ and
          $\beta\in[1,k+\frac{8}{3})$, then there exists  $C(L,k, \beta)>0$ such that
         \begin{align}\label{1.16}
             \int_{\Omega}\varphi^{\beta}\psi\leq  \int_{\Omega}\varphi^{k}\psi|\nabla \varphi|^2+\int_{\Omega}\varphi\frac{|\nabla \psi|^4}{\psi^3}+C(L,k, \beta)\int_{\Omega}\varphi \psi
         \end{align}
          holds for all sufficiently regular positive functions $\varphi$ and $\psi$ fulfilling $\int_{\Omega}\varphi\leq L$ (see Lemma \ref{lemma3.5}), which is somewhat different  from that in \cite{Wu}. Thanks to the estimates \eqref{1.15}, we then employ inequality \eqref{1.16} to achieve the desired $L^{p_0}$ bounds for $u$. This improved integrability of $u$ allows us to derive a refined version of \eqref{1.16} (see Lemma \ref{lemma3.7}) and accordingly  certify an energy-like property for the coupled quantity
         \begin{align*}
           \mathcal{ H}(t):=\int_{\Omega}\frac{|\nabla v(\cdot,t)|^q}{v^{q-1}(\cdot,t)}+\int_{\Omega}u^{p}(\cdot,t)
        \end{align*}
        for any fixed $p>1$.
         This property thereby yields the bounds for $\|u(\cdot,t)\|_{L^{p}(\Omega)}$ as detailed in Lemma \ref{lemma3.8}.

By suitably taking into account respective dependence  on time,
 we intend to characterize the large-time behavior of solution $(u,v)$ to system \eqref{1.4}. To achieve this, we employ a duality-based estimate for the time derivative $u_t$, establishing that
\begin{align}\label{1.18}
    \int_0^{\infty}\|u_t(\cdot,t)\|_{(W^{1,\infty}(\Omega))^*}dt\leq C(K)\cdot\left\{ \int_{\Omega} v_0\right\}^\sigma,
\end{align}
for some  $C(K)>0$  and $\sigma>0$ (see Lemma \ref{lemma4.1}). Thanks to  \eqref{1.18}, we prove that the solution of \eqref{1.4} converges to equilibrium state $(u_{\infty},0)$ as $t\to \infty$ (see Lemma \ref{lemma4.3} and Lemma \ref{lemma4.5}). Moreover, we  establish that for sufficiently small initial data $v_0$, the long-time limit $u_{\infty}$
  exhibits spatial heterogeneity (Lemma \ref{lemma4.6}).


The structure of this paper is as follows: In Section 2,  we specify the weak solutions of system  \eqref{1.4} and
 establishes fundamental a priori estimates for the regularized systems.
 In Section 3, crucial $L^p$-estimates for $u_\varepsilon$ are derived  through careful analysis of energy functionals and bootstrap
  arguments, which  needs  to effectively quantifying  the interplay between the weaken smoothing effects of random diffusion  and the potentially destabilizing influence of taxis-diffusion.
    Building upon these estimates and further higher regularity properties of the regularized systems, we then
    prove Theorem 1.1 via compactness arguments. In Section 4, we investigates the stability properties of $(u,v)$ by means of a duality-based argument.
\section{Preliminaries}\label{sec:preli}
In  view of the considered diffusion
degeneracy in \eqref{1.4}, the notation of  weak solutions  specified below seems  fairly natural (\cite{WJDE}), but it is somewhat
stronger  than those in \cite{LiG,WCVPDE}, as  it requires the integrability of $\nabla u$.

\begin{definition}\label{Definition 2.1} Assume  that \eqref{1.5} and \eqref{1.6} hold, suppose that $\chi >0$ and $\ell> 0$, and  that $\alpha\geq 1$. Then a pair of nonnegative functions
\begin{equation}\label{2.1}
\left\{
\begin{aligned}
&u\in C^{0}(\overline{\Omega}\times [0,\infty))\quad and \\
&v\in C^{0}(\overline{\Omega}\times [0,\infty))\cap C^{2,1}(\overline{\Omega}\times(0,\infty))\\
\end{aligned}
\right.
\end{equation}
such that
\begin{align}\label{2.2}
  u^2 \in L^{1}_{loc}([0,\infty); W^{1,1}(\Omega))\quad \hbox{and}\quad u^\alpha\nabla v \in L^{1}_{loc}(\overline{\Omega}\times
  [0,\infty);\mathbb{R}^3)
\end{align}
 will be called a continuous weak solution of \eqref{1.4} if
\begin{align}\label{2.3}
    -\int_{0}^{\infty}\!\int_{\Omega}u\varphi_t-\int_{\Omega}u_0\varphi(\cdot,0)=
    &-\frac{1}{2}\int_{0}^{\infty} \!\int_{\Omega}
     v \nabla u^2 \cdot\nabla\varphi 
     +\chi\int_{0}^{\infty}\!\int_{\Omega}u^{\alpha}v\nabla v\cdot \nabla \varphi
     +\ell\int_{0}^{\infty}\!\int_{\Omega}uv\varphi
\end{align}
for all
 $\varphi \in C_0^{\infty}(\overline{\Omega}\times [0,\infty))$
 fulfilling $\frac{\partial \varphi}{\partial \nu}=0$ on $\partial\Omega\times(0,\infty)$, and if
\begin{align}\label{2.4}
\int_{0}^{\infty}\int_{\Omega}v\varphi_t+\int_{\Omega}v_0\varphi(\cdot,0)=\int_{0}^{\infty}\int_{\Omega}\nabla
v\cdot\nabla\varphi+\int_{0}^{\infty}\int_{\Omega}uv\varphi.
\end{align}
for any $\varphi \in C_0^{\infty}(\overline{\Omega}\times [0,\infty))$.
\end{definition}

To construct the defined weak solutions above through approximation by the classical solutions to the conveniently regularized problems, we consider the the
regularized variants of \eqref{1.4} given by
\begin{equation}\label{2.5}
\left\{
\begin{aligned}
&u_{\varepsilon t}=\nabla \cdot (u_{\varepsilon}v_{\varepsilon}\nabla u_{\varepsilon})-\chi \nabla \cdot
(u^{\alpha}_{\varepsilon}v_{\varepsilon}\nabla v_{\varepsilon})+\ell u_{\varepsilon}v_{\varepsilon},&x\in \Omega,\, t>0,\\
& v_{\varepsilon t}=\Delta v_{\varepsilon}-u_{\varepsilon}v_{\varepsilon},&x\in \Omega,\, t>0,\\
& \frac{\partial u_{\varepsilon }}{\partial \nu}=\frac{\partial v_{\varepsilon }}{\partial \nu}=0,&x\in \partial\Omega,\, t>0,\\
& u_{\varepsilon }(x,0)=u_0(x)+\varepsilon,\;v_{\varepsilon }(x,0)=v_0(x),&x\in \Omega,\,
\end{aligned}
\right.
\end{equation}
for $\varepsilon\in(0,1)$. According to Lemma 2.1 of \cite{WNARWA}, we have the following statement regarding the existence, extensibility and  basic
properties of solutions to \eqref{2.5}.

\begin{lemma}\label{lemma2.1}
Let $\Omega \subset \mathbb R^3$ be a bounded domain with smooth boundary,
$\chi >0$, $\ell>0$ and   $\alpha\geq 1$, and assume that \eqref{1.5} and \eqref{1.6} hold.  Then for each $\varepsilon\in(0,1)$, there exist
$T_{max,\varepsilon}\in(0,\infty]$ and functions
\begin{equation}\label{2.6}
\left\{
\begin{aligned}
&u_{\varepsilon }\in C^0(\overline{\Omega}\times [0,T_{max,\varepsilon})\cap C^{2,1}(\overline{\Omega}\times(0,T_{max,\varepsilon})) \\
&v_{\varepsilon }\in  C^0(\overline{\Omega}\times[0,T_{max,\varepsilon}))\cap C^{2,1}(\overline{\Omega}\times(0,T_{max,\varepsilon})) \\
\end{aligned}
\right.
\end{equation}
such that $u_ \varepsilon >0$ and $v_{\varepsilon}>0$ in $\overline{\Omega} \times [0,T_{max,\varepsilon})$, that $(u_{\varepsilon
},v_{\varepsilon })$ solves \eqref{2.5} classically in $\Omega\times(0,T_{max,\varepsilon})$, and that
  \begin{equation} \label{2.7}
   \hbox{if}~~  T_{max,\varepsilon}<\infty, ~~ \hbox{then}~
   \limsup_{t\nearrow T_{max,\varepsilon}}\|u(\cdot , t)\|_{L^{\infty}(\Omega)}=\infty.
\end{equation}
Furthermore, this solution satisfies
\begin{align}\label{2.8}
    \|v_{\varepsilon}(\cdot , t)\|_{L^{\infty}(\Omega)}\leq \|v_{0 }\|_{L^{\infty}(\Omega)}\qquad for\; all\; t\in(0,T_{max,\varepsilon})
\end{align}
and
\begin{align}\label{2.9}
    \int_{\Omega}u_{0}\leq \int_{\Omega}u_{\varepsilon}(\cdot,t)\leq \int_{\Omega}u_{0}+\ell\int_{\Omega}v_{0}+|\Omega|\qquad for\;
    all\; t\in(0,T_{max,\varepsilon})
\end{align}
as well as
\begin{align}\label{2.10}
    \int_{0}^{T_{max,\varepsilon}}\int_{\Omega}u_{\varepsilon }v_{\varepsilon }\leq \int_{\Omega}v_{0}.
\end{align}
\end{lemma}

Throughout the sequel, unless explicitly stated otherwise, we shall assume  that
$\Omega\subset \mathbb R^3$  is a smoothly bounded domain (the convexity of $\Omega$ is not required),
$\chi >0$, $\ell> 0$ and   $\alpha\geq 1$. Moreover, whenever  functions  $u_0$ and $v_0$ fulfilling \eqref{1.5}--\eqref{1.6}
have been selected,  we shall  denote by $(u_{\varepsilon
},v_{\varepsilon })$  the solutions given by Lemma 2.1, with $T_{max,\varepsilon}$ representing their maximal existence time.

Due to the nonnegativity of $u_{\varepsilon} $ and $v_{\varepsilon} $, the evolution of $v_{\varepsilon}$ yields the following estimate.
\begin{lemma}\label{lemma2.2}For any $K>0$ with the property that \eqref{1.6} is valid, one can
find $C=C(K)>0$ such that
\begin{align}\label{2.11}
    \int_0^{T_{max,\varepsilon}}\int_{\Omega}v_{\varepsilon}|\nabla v_{\varepsilon}|^2\leq C.
\end{align}
\end{lemma}
\begin{proof}
    Multiplying the second equation in \eqref{2.5} by $v_{\varepsilon}^2$, we have
    \begin{align}\label{2.12}
    \frac{d}{dt}\int_{\Omega}v_{\varepsilon}^3+6\int_{\Omega}v_{\varepsilon}|\nabla
    v_{\varepsilon}|^2=-3\int_{\Omega}u_{\varepsilon}v_{\varepsilon}^3\leq 0\qquad for\;
    all\; t\in(0,T_{max,\varepsilon}).
\end{align}
Integrating \eqref{2.12} over $(0,T_{max,\varepsilon})$ and using \eqref{1.5}, we then obtain \eqref{2.11}.
\end{proof}
It is noticed  that in  \cite{LiG} 
the bounds of $\int_0^{T_{\max, \varepsilon}}\int_{\Omega} u^{2-2\alpha}_{\varepsilon} v_{\varepsilon}|\nabla
u_{\varepsilon }|^2$ was utilized as a basic piece of information in  deriving   the bounds for $\int_\Omega u^{p_0}_\varepsilon(\cdot,t)$ with some $p_0>1$.
In contrast, our approach in this paper takes
$ \int_0^{T_{\max, \varepsilon}}\int_{\Omega} u_{\varepsilon}^{1-\alpha}v_{\varepsilon}|\nabla  u_{\varepsilon }|^2$
as the starting point for the bootstrap argument, which is achieved  by analyzing the temporal evolution of $
-\frac{1}{2-\alpha}\int_{\Omega} u_{\epsilon}^{2-\alpha}+
  \frac{\chi(\alpha-1)}{2} \int_{\Omega} |\nabla v_{\varepsilon }|^2$ rather
$-\frac{1}{3-2\alpha}\int_{\Omega}u_{\varepsilon}^{3-2\alpha}$.
\begin{lemma}\label{lemma2.3} Let $\alpha \in\left(1,2\right]$. Then for all $K>0$ with the property that
 \eqref{1.6} is valid, one can find $C=C(K)>0$ such that
\begin{align}\label{2.13}
\int_0^{T_{\max, \varepsilon}}\int_{\Omega} u_{\varepsilon}^{1-\alpha}v_{\varepsilon}|\nabla  u_{\varepsilon }|^2 \leq C  .
\end{align}
\end{lemma}

\begin{proof}
Multiplying the equations in \eqref{2.5} by $u_{\epsilon}^{1-\alpha}$ and $\Delta v_{\varepsilon }$ respectively, we obtain
\begin{align}\label{2.14}
    -\frac{1}{2-\alpha}\frac{d}{dt}\int_{\Omega}u_{\varepsilon}^{2-\alpha}+(\alpha-1)\int_{\Omega}
    u_{\varepsilon}^{1-\alpha}v_{\varepsilon}|\nabla  u_{\varepsilon }|^2=(\alpha-1)\chi\int_{\Omega} v_{\varepsilon}\nabla
    u_{\varepsilon}\cdot \nabla v_{\varepsilon}-\ell\int_{\Omega} u_{\varepsilon}^{2-\alpha}v_{\varepsilon}
\end{align}
and
\begin{align}\label{2.15}
    \frac{1}{2}\frac{d}{dt}\int_{\Omega}|\nabla v_{\varepsilon}|^2+\int_{\Omega}|\Delta
    v_{\varepsilon}|^2=\int_{\Omega}u_{\varepsilon}v_{\varepsilon}\Delta v_{\varepsilon}=-\int_{\Omega} v_{\varepsilon}\nabla
    u_{\varepsilon}\cdot \nabla v_{\varepsilon}-\int_{\Omega}u_{\varepsilon}|\nabla v_{\varepsilon}|^2
\end{align}
for all $t \in(0, T_{\max, \varepsilon})$. From \eqref{2.14} and \eqref{2.15}, we have
\begin{align*}
&\frac{d}{dt}\left\{-\frac{1}{2-\alpha}\int_{\Omega} u_{\epsilon}^{2-\alpha}+
  \frac{\chi(\alpha-1)}{2} \int_{\Omega} |\nabla v_{\varepsilon }|^2\right\}+(\alpha-1)\int_{\Omega}
    u_{\varepsilon}^{1-\alpha}v_{\varepsilon}|\nabla  u_{\varepsilon }|^2\\
&\leq -\ell\int_{\Omega} u_{\varepsilon}^{2-\alpha}v_{\varepsilon}-\chi(\alpha-1)\int_{\Omega}u_{\varepsilon}|\nabla v_{\varepsilon}|^2
\end{align*}
for all $t \in(0, T_{\max, \varepsilon})$. Integrating the above differential inequality on $(0,t)$, along with \eqref{1.6} and \eqref{2.9} we then get
\begin{align*}
& \frac{1}{2-\alpha}\int_{\Omega} u_{0\epsilon}^{2-\alpha}+
  \frac{\chi(\alpha-1)}{2} \int_{\Omega} |\nabla v_{\varepsilon }|^2
   +(\alpha-1)\int_0^{t}\int_{\Omega}
    u_{\varepsilon}^{1-\alpha}v_{\varepsilon}|\nabla  u_{\varepsilon }|^2\nonumber\\
    &+\ell\int_0^{t}\int_{\Omega} u_{\varepsilon}^{2-\alpha}v_{\varepsilon}+\chi(\alpha-1)\int_0^{t}\int_{\Omega}u_{\varepsilon}|\nabla
   v_{\varepsilon}|^2\nonumber \\
   = &\frac{1}{2-\alpha}\int_{\Omega} u_{\varepsilon}^{2-\alpha}+
  \frac{\chi(\alpha-1)}{2} \int_{\Omega} |\nabla v_{0 }|^2\nonumber\\
   \leq &\frac{1}{2-\alpha}\int_{\Omega}\left(u_{\varepsilon}+1\right)+
  \frac{\chi(\alpha-1)}{2} \int_{\Omega} |\nabla v_{0 }|^2\nonumber\\
   \leq &C(K)
\end{align*}
for all $t \in(0, T_{\max, \varepsilon})$, due to $2-\alpha\in [0,1)$, which implies \eqref{2.13} immediately.
\end{proof}

When the  convexity of $\Omega$ is not assumed, the following estimate indicates  how the related boundary integral over $\partial \Omega$ can be controlled by the corresponding higher-order integrals on $\Omega$ involving singular weights.
\begin{lemma}\label{lemma2.4}(\cite{WDCDSB})
    Let $q\geq2$ and $\eta>0$. Then there exists $C=C(\eta,q)>0$ with the property that
    \begin{align}\label{2.16}
        \int_{\partial\Omega}\psi^{-q+1}|\nabla\psi|^{q-2}\cdot\frac{\partial|\nabla \psi|}{\partial\nu}\leq \eta
        \int_{\Omega}\psi^{-q+1}|\nabla\psi|^{q-2}|D^2\psi|^2+\eta\int_{\Omega}\psi^{-q-1}|\nabla\psi|^{q+2}+C\int_{\Omega}\psi
    \end{align}
    for any $\psi\in C^2(\overline{\Omega})$ which is such that $\psi>0$ in $\overline{\Omega}$ and $\frac{\partial\psi}{\partial\nu}=0$ on
    $\partial\Omega$.
\end{lemma}
By analyzing the temporal evolution of
 $$\|v_0\|_{L^{\infty}(\Omega)}^2\int_{\Omega} \frac{|\nabla v_{\varepsilon}|^2}{v_{\varepsilon}}-\int_{\Omega}\ln{u_{\varepsilon}},$$
 we establish spatio-temporal integral bounds for both $v_{\varepsilon}$ and
$\frac{u_{\varepsilon}}{v_{\varepsilon}}|\nabla  v_{\varepsilon }|^2$,  thanks to  the
   estimates \eqref{2.13} and \eqref{2.16}.
\begin{lemma}\label{lemma2.5}
Assume $\alpha\in\left(\frac{3}{2},2\right]$ and $K>0$ such
that \eqref{1.6} holds. Then  there exists $C>0$ independent of $T_{\max, \varepsilon}$  such that
\begin{align}\label{2.18}
\int_0^{T_{\max, \varepsilon}}\int_{\Omega}v_{\varepsilon}\leq C
\end{align}
and
\begin{align}\label{2.19}
\int_0^{T_{\max, \varepsilon}}\int_{\Omega} \frac{u_{\varepsilon}}{v_{\varepsilon}}|\nabla  v_{\varepsilon }|^2 \leq C
\end{align}
as well as
\begin{align}\label{2.19-1}
\int_0^{T_{\max, \varepsilon}}\int_{\Omega} \frac{v_{\varepsilon}}{u_{\varepsilon}}|\nabla  u_{\varepsilon }|^2 \leq C\qquad and \qquad\int_0^{T_{\max, \varepsilon}}\int_{\Omega} \frac{|\nabla  v_{\varepsilon }|^4}{v_{\varepsilon}^3} \leq C.
\end{align}
\end{lemma}
\begin{proof} According  to the second equation in \eqref{2.5}, several integrations by parts show that
for all $t \in\left(0, T_{\max, \varepsilon}\right)$,
\begin{align*}
\frac{1}{2} \frac{d}{dt} \int_{\Omega} \frac{\left|\nabla v_{\varepsilon}\right|^2}{v_{\varepsilon}}
= & \int_{\Omega} \frac{1}{v_{\varepsilon}} \nabla v_{\varepsilon} \cdot \nabla\left\{\Delta v_{\varepsilon}
    -u_{\varepsilon} v_{\varepsilon}\right\}
    -\frac{1}{2} \int_{\Omega} \frac{1}{v_{\varepsilon}^2}\left|\nabla v_{\varepsilon}\right|^2 \cdot\left\{\Delta
    v_{\varepsilon}-u_{\varepsilon} v_{\varepsilon}\right\} \nonumber\\
= & \int_{\Omega} \frac{1}{v_{\varepsilon}} \nabla v_{\varepsilon} \cdot \nabla\Delta v_{\varepsilon}
    -\int_{\Omega} \frac{u_{\varepsilon}}{v_{\varepsilon}}\left|\nabla v_{\varepsilon}\right|^2
    -\int_{\Omega} \nabla u_{\varepsilon} \cdot \nabla v_{\varepsilon}\nonumber\\
  & -\frac{1}{2} \int_{\Omega} \frac{1}{v_{\varepsilon}^2}\left|\nabla v_{\varepsilon}\right|^2 \cdot\Delta v_{\varepsilon}
    +\frac{1}{2} \int_{\Omega} \frac{u_{\varepsilon}}{v_{\varepsilon}}\left|\nabla v_{\varepsilon}\right|^2\nonumber\\
= & \int_{\Omega} \frac{1}{v_{\varepsilon}} \nabla v_{\varepsilon} \cdot \nabla\Delta v_{\varepsilon}
    -\frac{1}{2} \int_{\Omega} \frac{1}{v_{\varepsilon}^2}\left |\nabla v_{\varepsilon}\right|^2 \cdot\Delta v_{\varepsilon}-\frac{1}{2}
    \int_{\Omega} \frac{u_{\varepsilon}}{v_{\varepsilon}}\left|\nabla v_{\varepsilon}\right|^2
    -\int_{\Omega} \nabla u_{\varepsilon} \cdot \nabla v_{\varepsilon}\nonumber\\
 = & \frac{1}{2}\int_{\Omega} \frac{1}{v_{\varepsilon}}\Delta |\nabla v_{\varepsilon}|^2-\int_{\Omega} \frac{1}{v_{\varepsilon}}|D^2
 v_{\varepsilon}|^2+ \frac{1}{2}\int_{\Omega} \nabla ( \frac{|\nabla v_{\varepsilon}|^2}{v_{\varepsilon}^2} ) \cdot\nabla
 v_{\varepsilon}\nonumber\\
 &-\frac{1}{2} \int_{\Omega} \frac{u_{\varepsilon}}{v_{\varepsilon}}\left|\nabla v_{\varepsilon}\right|^2
    -\int_{\Omega} \nabla u_{\varepsilon} \cdot \nabla v_{\varepsilon}\nonumber\\
 = &\frac{1}{2} \int_{\partial \Omega} \frac{1}{v_{\varepsilon}} \frac{\partial|\nabla v_{\varepsilon}|^2}{\partial
 \nu}+\int_{\Omega}\frac{1}{v_{\varepsilon}^2} \nabla v_{\varepsilon}\cdot \nabla |\nabla v_{\varepsilon}|^2 -\int_{\Omega}
 \frac{1}{v_{\varepsilon}}|D^2 v_{\varepsilon}|^2-\int_{\Omega} \frac{|\nabla v_{\varepsilon}|^4}{ v_{\varepsilon}^3}\nonumber\\
 &-\frac{1}{2} \int_{\Omega} \frac{u_{\varepsilon}}{v_{\varepsilon}}\left|\nabla v_{\varepsilon}\right|^2
    -\int_{\Omega} \nabla u_{\varepsilon} \cdot \nabla v_{\varepsilon},
\end{align*}
 where  $\nabla v_{\varepsilon} \cdot \nabla \Delta
v_{\varepsilon}=\frac{1}{2}\Delta|\nabla v_{\varepsilon}|^2-|D^2 v_{\varepsilon}|^2$ is used. Thanks to the identity (\cite[Lemma 3.2]{WDCDSB})
 \begin{align}\label{2.20}
    |D^2 v_{\varepsilon}|^2=v_{\varepsilon}^2|D^2 \ln v_{\varepsilon}|^2+\frac{1}{v_{\varepsilon}}\nabla v_{\varepsilon}\cdot\nabla|\nabla
    v_{\varepsilon}|^2-\frac{1}{v_{\varepsilon}^2}|\nabla v_{\varepsilon}|^4,
    \end{align}
we then  get
\begin{align}\label{2.21}
\frac{1}{2} \frac{d}{dt}\int_{\Omega} \frac{\left|\nabla v_{\varepsilon}\right|^2}{v_{\varepsilon}}
  +\int_{\Omega} v_{\varepsilon}\left|D^2 \ln v_{\varepsilon}\right|^2
    +\frac{1}{2} \int_{\Omega} \frac{u_{\varepsilon}}{v_{\varepsilon}}\left|\nabla v_{\varepsilon}\right|^2
   =  \frac{1}{2} \int_{\partial \Omega} \frac{1}{v_{\varepsilon}} \frac{\partial|\nabla v_{\varepsilon}|^2}{\partial \nu}
    -\int_{\Omega} \nabla u_{\varepsilon} \cdot \nabla v_{\varepsilon}
\end{align}
for all $t \in\left(0, T_{\max, \varepsilon}\right)$. From Lemma 3.4 in \cite{WDCDSB}, it follows that there
exists  $c_1>0$ such that
\begin{equation}\label{2.22}
\int_{\Omega} v_{\varepsilon}\left|D^2 \ln v_{\varepsilon}\right|^2
\geq c_1 \int_{\Omega} \frac{\left|D^2 v_{\varepsilon}\right|^2}{v_{\varepsilon}}
  +c_1 \int_{\Omega} \frac{\left|\nabla v_{\varepsilon}\right|^4}{v_{\varepsilon}^3}.
\end{equation}
As an application of Lemma \ref{lemma2.4} with $\eta:=c_1, q:=2$, we arrive at
\begin{align}\label{2.23}
 \int_{\partial \Omega} \frac{1}{v_{\varepsilon}} \cdot \frac{\partial\left|\nabla v_{\varepsilon}\right|^2}{\partial \nu}
\leq c_1
\int_{\Omega} \frac{\left|D^2 v_{\varepsilon}\right|^2}{v_{\varepsilon}}
  +c_1 \int_{\Omega} \frac{\left|\nabla v_{\varepsilon}\right|^4}{v_{\varepsilon}^3}
  +C(c_1,2) \int_{\Omega} v_{\varepsilon}.
\end{align}
Apart from that, by Young's inequality, we infer that for all $t
\in\left(0, T_{\max, \varepsilon}\right)$
\begin{align}\label{2.24}
-\int_{\Omega} \nabla u_{\varepsilon} \cdot \nabla v_{\varepsilon}
&\leq \int_{\Omega}u_{\varepsilon}^{1-\alpha}v_{\varepsilon}|\nabla u_{\varepsilon }|^2
+\frac{1}{4}\int_{\Omega} \frac{u_{\varepsilon}^{\alpha-1}}{v_{\varepsilon}} |\nabla v_{\varepsilon }|^2\nonumber\\
&\leq \int_{\Omega}u_{\varepsilon}^{1-\alpha}v_{\varepsilon}|\nabla u_{\varepsilon }|^2
+\frac{1}{4}\int_{\Omega} \left(u_{\varepsilon}+1\right)\frac{|\nabla v_{\varepsilon }|^2}{v_{\varepsilon}}.\nonumber\\
&\leq  \int_{\Omega}u_{\varepsilon}^{1-\alpha}v_{\varepsilon}|\nabla u_{\varepsilon }|^2
+\frac{1}{4}\int_{\Omega} \frac{u_{\varepsilon}}{v_{\varepsilon}}|\nabla v_{\varepsilon }|^2+\frac{c_1}{4}\int_{\Omega} \frac{|\nabla v_{\varepsilon }|^4}{v_{\varepsilon}^3}+ \frac{1}{2c_1} \int_{\Omega} v_{\varepsilon}.
\end{align}
Summing up \eqref{2.21}--\eqref{2.24}, we conclude that
\begin{align}\label{2.25}
\frac{d}{dt}\int_{\Omega} \frac{|\nabla v_{\varepsilon}|^2}{v_{\varepsilon}}
+\frac{1}{2} \int_{\Omega} \frac{u_{\varepsilon}}{v_{\varepsilon}} |\nabla v_{\varepsilon }|^2+\frac{c_1}{2}\int_{\Omega} \frac{|\nabla v_{\varepsilon }|^4}{v_{\varepsilon}^3}
\leq  2 \int_{\Omega}u_{\varepsilon}^{1-\alpha}v_{\varepsilon}|\nabla u_{\varepsilon }|^2
+c_2 \int_{\Omega} v_{\varepsilon}
\end{align}
with $ c_2 :=C(c_1,2)+\frac{1}{c_1}$ for all $t \in(0, T_{\max, \varepsilon})$. On the other hand, multiplying the first equation in \eqref{2.5} by $-\frac{1}{u_{\varepsilon}}$ and using \eqref{2.8} and Young's inequality, we have
\begin{align}\label{2.26}
&-\frac{d}{dt}\int_{\Omega}\ln{u_{\varepsilon}}+\ell \int_{\Omega} v_{\varepsilon}+\int_{\Omega} \frac{v_{\varepsilon}}{u_{\varepsilon}} |\nabla u_{\varepsilon }|^2\nonumber\\
&=\chi \int_{\Omega}u_{\varepsilon}^{\alpha-2}v_{\varepsilon}\nabla u_{\varepsilon}\cdot \nabla v_{\varepsilon}\nonumber\\
&\leq \frac{1}{2}\int_{\Omega} \frac{v_{\varepsilon}}{u_{\varepsilon}} |\nabla u_{\varepsilon }|^2 +\frac{\chi^2}{2}\int_{\Omega} u_{\varepsilon}^{2\alpha-3}v_{\varepsilon} |\nabla v_{\varepsilon }|^2\nonumber\\
 &\leq\frac{1}{2}\int_{\Omega} \frac{v_{\varepsilon}}{u_{\varepsilon}} |\nabla u_{\varepsilon }|^2 +\frac{\chi^2}{2}\int_{\{u_{\varepsilon}\geq 1\}} u_{\varepsilon}^{2\alpha-3}v_{\varepsilon} |\nabla v_{\varepsilon }|^2+\frac{\chi^2}{2}\int_{\{u_{\varepsilon}<1\}} v_{\varepsilon} |\nabla v_{\varepsilon }|^2\nonumber\\
 &\leq\frac{1}{2}\int_{\Omega} \frac{v_{\varepsilon}}{u_{\varepsilon}} |\nabla u_{\varepsilon }|^2 +\frac{\chi^2\eta}{2}\int_{\Omega} u_{\varepsilon}v_{\varepsilon} |\nabla v_{\varepsilon }|^2+\frac{\chi^2}{2}(\eta^{-1})^{\frac{2\alpha-3}{4-2\alpha}}\int_{\Omega} v_{\varepsilon} |\nabla v_{\varepsilon }|^2+\frac{\chi^2}{2}\int_{\{u_{\varepsilon}<1\}} v_{\varepsilon} |\nabla v_{\varepsilon }|^2\nonumber\\
 &\leq\frac{1}{2}\int_{\Omega} \frac{v_{\varepsilon}}{u_{\varepsilon}} |\nabla u_{\varepsilon }|^2 +\frac{\chi^2\eta}{2}\|v_0\|_{L^{\infty}(\Omega)}^2\int_{\Omega} \frac {u_{\varepsilon}}{v_{\varepsilon}} |\nabla v_{\varepsilon }|^2+\frac{\chi^2}{2}(\eta^{\frac{3-2\alpha}{4-2\alpha}}+1)\int_{\Omega} v_{\varepsilon} |\nabla v_{\varepsilon }|^2
\end{align}
for all $t \in(0, T_{\max, \varepsilon})$, due to $\alpha\in\left(\frac{3}{2},2\right]$. Letting
$$
\eta:=\frac{\ell}{4\chi^2c_2\|v_0\|_{L^{\infty}(\Omega)}^2} \quad \text{and} \quad C(\alpha):=\frac{\chi^2}{2}(\eta^{\frac{3-2\alpha}{4-2\alpha}}+1).
$$
Then  combining \eqref{2.25} with  \eqref{2.26} yields
\begin{align*}
    &\frac{d}{dt}\left\{\frac{\ell}{2c_2}\int_{\Omega} \frac{|\nabla v_{\varepsilon}|^2}{v_{\varepsilon}}-\int_{\Omega}\ln{u_{\varepsilon}} \right\}+\frac{\ell}{4c_2} \int_{\Omega} \frac{u_{\varepsilon}}{v_{\varepsilon}} |\nabla v_{\varepsilon }|^2+\frac{\ell c_1}{4c_2}\int_{\Omega} \frac{|\nabla v_{\varepsilon }|^4}{v_{\varepsilon}^3}+\frac{\ell}{2}\int_{\Omega} v_{\varepsilon}+\frac{1}{2}\int_{\Omega} \frac{v_{\varepsilon}}{u_{\varepsilon}} |\nabla u_{\varepsilon }|^2\\
    &\leq \frac{\chi^2\eta}{2}\|v_0\|_{L^{\infty}(\Omega)}^2\int_{\Omega} \frac {u_{\varepsilon}}{v_{\varepsilon}} |\nabla v_{\varepsilon }|^2 +\frac{\ell}{c_2} \int_{\Omega}u_{\varepsilon}^{1-\alpha}v_{\varepsilon}|\nabla u_{\varepsilon }|^2+C(\alpha)\int_{\Omega} v_{\varepsilon} |\nabla v_{\varepsilon }|^2\\
    &= \frac{\ell}{8c_2} \int_{\Omega} \frac{u_{\varepsilon}}{v_{\varepsilon}} |\nabla v_{\varepsilon }|^2 +\frac{\ell}{c_2} \int_{\Omega}u_{\varepsilon}^{1-\alpha}v_{\varepsilon}|\nabla u_{\varepsilon }|^2+C(\alpha)\int_{\Omega} v_{\varepsilon} |\nabla v_{\varepsilon }|^2
\end{align*}
for all $t \in(0, T_{\max, \varepsilon})$. On integration in time, according to $\ln{\xi}\leq \xi$ for all $\xi>0$ entails that
\begin{align}\label{2.27}
   &\frac{\ell}{2c_2}\int_{\Omega} \frac{|\nabla v_{\varepsilon}|^2}{v_{\varepsilon}}+\frac{\ell}{8c_2} \int_0^t\int_{\Omega} \frac{u_{\varepsilon}}{v_{\varepsilon}} |\nabla v_{\varepsilon }|^2+\frac{\ell c_1}{4c_2}\int_0^t\int_{\Omega} \frac{|\nabla v_{\varepsilon }|^4}{v_{\varepsilon}^3}+\frac{\ell}{2}\int_0^t\int_{\Omega} v_{\varepsilon}+\frac{1}{2}\int_0^t\int_{\Omega} \frac{v_{\varepsilon}}{u_{\varepsilon}} |\nabla u_{\varepsilon }|^2\nonumber\\
    &\leq -\int_{\Omega}\ln{u_{0\varepsilon}}+\int_{\Omega}\ln{u_{\varepsilon}}+\frac{\ell}{2c_2}\int_{\Omega} \frac{|\nabla v_{0\varepsilon}|^2}{v_{0\varepsilon}}+ \frac{\ell}{c_2} \int_0^t\int_{\Omega}u_{\varepsilon}^{1-\alpha}v_{\varepsilon}|\nabla u_{\varepsilon }|^2+C(\alpha)\int_0^t\int_{\Omega} v_{\varepsilon} |\nabla v_{\varepsilon }|^2\nonumber\\
    &\leq -\int_{\Omega}\ln{u_0}+\int_{\Omega}u_{\varepsilon}+\frac{\ell}{2c_2}\int_{\Omega} \frac{|\nabla v_{0\varepsilon}|^2}{v_{0\varepsilon}}+ \frac{\ell}{c_2} \int_0^t\int_{\Omega}u_{\varepsilon}^{1-\alpha}v_{\varepsilon}|\nabla u_{\varepsilon }|^2+C(\alpha)\int_0^t\int_{\Omega} v_{\varepsilon} |\nabla v_{\varepsilon }|^2
\end{align}
for all $t \in(0, T_{\max, \varepsilon})$. Therefore,
 \eqref{2.27}  implies \eqref{2.18}-\eqref{2.19-1} in  view of \eqref{1.6}, \eqref{2.9}, \eqref{2.11} and \eqref{2.13}.
\end{proof}

With the help of the following class of interpolation inequalities, the results derived from \eqref{2.13} and \eqref{2.19} can be used to estimate $\int_0^{T_{\max}}\int_\Omega u_{\varepsilon}^{\frac{5}{3}}v_{\varepsilon}$ in the next section.
\begin{lemma}\label{lemma2.6}(\cite{WCVPDE})
Let $p^*\geq1$ and $p:=\frac{2p^*+3}{3}$. Then for all $M>0$ one can find $C(M,p^*)>0$ such that for any $q\in[0,\frac{2p^*}{3}]$ and   functions $\phi,\psi\in C^1(\overline{\Omega})$ satisfying $\phi>0,\psi>0$ in $\overline{\Omega}$ as well as
\begin{align*}
    \int_{\Omega}\phi^{p^*}\leq M,
\end{align*}
we have
\begin{align}\label{eq2.28}
    \int_{\Omega}\phi^{p}\psi\leq C(M,p^*) \left\{\int_{\Omega}\phi^{q-1}\psi|\nabla \phi|^2+\int_{\Omega}\frac{\phi}{\psi}|\nabla
    \psi|^2+\int_{\Omega}\phi\psi\right\}.
\end{align}
\end{lemma}

Since  the convexity requirement for the domain $\Omega$ is not assumed, we establish a refined results stated in Lemma 2.3 of \cite{WJDE}.
\begin{lemma}\label{lemma2.7}
Let $\Omega \subset \mathbb R^3$ be a bounded domain with smooth boundary and $q\geq 2$. Then for all $t\in (0,T_{max,\epsilon})$, there exist
$\Gamma(q)>0$ and $\gamma(q)>0$ such that
\begin{align}\label{2.29}
    \frac{d}{dt}\int_{\Omega}\frac{|\nabla v_{\epsilon}|^q}{v_{\epsilon}^{q-1}}+\Gamma(q)\int_{\Omega}\frac{|\nabla
    v_{\epsilon}|^{q-2}}{v_{\epsilon}^{q-3}}|D^2\ln{v_{\epsilon}}|^2+\Gamma(q)\int_{\Omega}u_{\epsilon}\frac{|\nabla
    v_{\epsilon}|^q}{v_{\epsilon}^{q-1}}\leq \frac{1}{\Gamma(q)}\int_{\Omega}v_{\epsilon}(u_{\epsilon}^{\frac{q+2}{2}}+1)
\end{align}
and
\begin{align}\label{2.30}
    \frac{d}{dt}\int_{\Omega}\frac{|\nabla v_{\varepsilon}|^q}{v_{\varepsilon}^{q-1}}+\gamma(q)\int_{\Omega}\frac{|\nabla
    v_{\varepsilon}|^{q-2}}{v_{\varepsilon}^{q-3}}|D^2\ln{v_{\varepsilon}}|^2+\gamma(q)\int_{\Omega}u_{\varepsilon}\frac{|\nabla
    v_{\varepsilon}|^q}{v_{\varepsilon}^{q-1}}\leq \frac{1}{\gamma(q)}\int_{\Omega}v_{\varepsilon}(|\nabla u_{\varepsilon}|^{\frac{q+2}{3}}+1).
\end{align}
\end{lemma}
\begin{proof}
According to the second equation in \eqref{2.5} and the identity $\nabla v_{\varepsilon} \cdot \nabla \Delta
v_{\varepsilon}=\frac{1}{2}\Delta|\nabla v_{\varepsilon}|^2-|D^2 v_{\varepsilon}|^2$, we can obtain that
    \begin{align}\label{2.31}
        \frac{d}{dt}\int_{\Omega}\frac{|\nabla v_{\varepsilon}|^q}{v_{\varepsilon}^{q-1}}=&q\int_{\Omega}v_{\varepsilon}^{-q+1}|\nabla
        v_{\varepsilon}|^{q-2}\nabla v_{\varepsilon}\cdot\left\{\nabla \Delta v_{\varepsilon}-\nabla (u_{\varepsilon}
        v_{\varepsilon})\right\}-(q-1)\int_{\Omega}v_{\varepsilon}^{-q}|\nabla v_{\varepsilon}|^q \cdot\left\{\Delta
        v_{\varepsilon}-u_{\varepsilon} v_{\varepsilon}\right\}\nonumber
        \\
        =&q\int_{\Omega}v_{\varepsilon}^{-q+1}|\nabla v_{\varepsilon}|^{q-2}\nabla v_{\varepsilon}\cdot\nabla \Delta
        v_{\varepsilon}-q\int_{\Omega}v_{\varepsilon}^{-q+1}|\nabla v_{\varepsilon}|^{q-2}\nabla v_{\varepsilon}\cdot\nabla (u_{\varepsilon}
        v_{\varepsilon})\nonumber\\
        &-(q-1)\int_{\Omega}v_{\varepsilon}^{-q}|\nabla v_{\varepsilon}|^q \cdot\Delta
        v_{\varepsilon}+(q-1)\int_{\Omega}u_{\varepsilon}v_{\epsilon}^{-q+1}|\nabla v_{\varepsilon}|^{q}\nonumber\\
        =&\frac{q}{2}\int_{\Omega}v_{\varepsilon}^{-q+1}|\nabla v_{\varepsilon}|^{q-2}\Delta |\nabla
        v_{\varepsilon}|^2-q\int_{\Omega}v_{\varepsilon}^{-q+1}|\nabla v_{\varepsilon}|^{q-2}|D^2
        v_{\varepsilon}|^2-(q-1)\int_{\Omega}v_{\varepsilon}^{-q}|\nabla v_{\varepsilon}|^{q} \Delta v_{\varepsilon}\nonumber\\
        &-q\int_{\Omega}v_{\varepsilon}^{-q+1}|\nabla v_{\varepsilon}|^{q-2}\nabla v_{\varepsilon}\cdot\nabla (u_{\varepsilon}
        v_{\varepsilon})+(q-1)\int_{\Omega}u_{\varepsilon}v_{\varepsilon}^{-q+1}|\nabla v_{\varepsilon}|^{q}\nonumber\\
        =&\frac{q}{2}\int_{\Omega}v_{\varepsilon}^{-q+1}|\nabla v_{\varepsilon}|^{q-2}\cdot\frac{\partial|\nabla
        v_{\varepsilon}|}{\partial\nu}-\frac{q(q-2)}{4}\int_{\Omega}v_{\varepsilon}^{-q+1}|\nabla v_{\varepsilon}|^{q-4}|\nabla |\nabla
        v_{\varepsilon}|^2|^2\nonumber\\
        &-q\int_{\Omega}v_{\varepsilon}^{-q+1}|\nabla v_{\varepsilon}|^{q-2}|D^2
        v_{\varepsilon}|^2+q(q-1)\int_{\Omega}v_{\varepsilon}^{-q}|\nabla v_{\varepsilon}|^{q-2}\nabla |\nabla v_{\varepsilon}|^2\cdot\nabla
        v_{\varepsilon}\nonumber\\
        &-q(q-1)\int_{\Omega}v_{\varepsilon}^{-q-1}|\nabla v_{\varepsilon}|^{q+2}-q\int_{\Omega}v_{\varepsilon}^{-q+1}|\nabla
        v_{\varepsilon}|^{q-2}\nabla v_{\varepsilon}\cdot\nabla (u_{\varepsilon} v_{\epsilon})\nonumber\\
        &+(q-1)\int_{\Omega}u_{\varepsilon}v_{\varepsilon}^{-q+1}|\nabla v_{\varepsilon}|^{q}
    \end{align}
    for all $t \in(0, T_{\max, \varepsilon})$. Thanks to \eqref{2.16}, one can see that for any $\eta>0$, there exists $c_1=C(\eta,q)>0$  such that
    \begin{align}\label{2.32}
        \int_{\partial\Omega}v_{\varepsilon}^{-q+1}|\nabla v_{\varepsilon}|^{q-2}\cdot\frac{\partial|\nabla v_{\varepsilon}|}{\partial\nu}\leq
        \eta \int_{\Omega}v_{\varepsilon}^{-q+1}|\nabla
        v_{\varepsilon}|^{q-2}|D^2v_{\varepsilon}|^2+\eta\int_{\Omega}v_{\varepsilon}^{-q-1}|\nabla
        v_{\varepsilon}|^{q+2}+c_1\int_{\Omega}v_{\varepsilon}.
    \end{align}
    Furthermore, by two well-known inequalities (Lemma 3.4 of \cite{WDCDSB})
    \begin{align}\label{2.33}
    \int_{\Omega}v_{\varepsilon}^{-q-1}|\nabla v_{\varepsilon}|^{q+2}\leq (q+\sqrt{3})^2\int_{\Omega}v_{\varepsilon}^{-q+3}|\nabla
    v_{\varepsilon}|^{q-2}|D^2 \ln v_{\varepsilon}|^2
    \end{align}
    and
    \begin{align}\label{2.34}
    \int_{\Omega}v_{\varepsilon}^{-q+1}|\nabla v_{\varepsilon}|^{q-2}|D^2 v_{\varepsilon}|^2\leq
    (q+\sqrt{3}+1)^2\int_{\Omega}v_{\varepsilon}^{-q+3}|\nabla v_{\varepsilon}|^{q-2}|D^2 \ln v_{\varepsilon}|^2,
    \end{align}
    we can conclude  that
    \begin{align}\label{2.35}
        &-\frac{q(q-2)}{4}\int_{\Omega}v_{\varepsilon}^{-q+1}|\nabla v_{\varepsilon}|^{q-4}|\nabla |\nabla
        v_{\varepsilon}|^2|^2-q\int_{\Omega}v_{\varepsilon}^{-q+1}|\nabla v_{\varepsilon}|^{q-2}|D^2 v_{\varepsilon}|^2\nonumber\\
        &+q(q-1)\int_{\Omega}v_{\varepsilon}^{-q}|\nabla v_{\varepsilon}|^{q-2}\nabla |\nabla v_{\varepsilon}|^2\cdot\nabla
        v_{\varepsilon}-q(q-1)\int_{\Omega}v_{\varepsilon}^{-q-1}|\nabla v_{\varepsilon}|^{q+2}\nonumber\\
        &=-\frac{q(q-2)}{4}\int_{\Omega}v_{\varepsilon}^{-q+1}|\nabla v_{\varepsilon}|^{q-4}|\nabla |\nabla
        v_{\varepsilon}|^2|^2-q\int_{\Omega}v_{\varepsilon}^{-q+3}|\nabla v_{\varepsilon}|^{q-2}|D^2 \ln v_{\varepsilon}|^2\nonumber\\
        &+q(q-2)\int_{\Omega}v_{\varepsilon}^{-q}|\nabla v_{\varepsilon}|^{q-2}\nabla |\nabla v_{\varepsilon}|^2\cdot\nabla
        v_{\varepsilon}-q(q-2)\int_{\Omega}v_{\varepsilon}^{-q-1}|\nabla v_{\varepsilon}|^{q+2}\nonumber\\
        &\leq -q\int_{\Omega}v_{\varepsilon}^{-q+3}|\nabla v_{\varepsilon}|^{q-2}|D^2 \ln v_{\varepsilon}|^2\nonumber\\
        &-\frac{q(q-2)}{4}\int_{\Omega}v_{\varepsilon}^{-q+1}|\nabla v_{\varepsilon}|^{q-4}| \nabla |\nabla v_{\varepsilon}|^2 -
        \frac{2}{v_{\varepsilon}}|\nabla v_{\varepsilon}|^2\nabla v_{\varepsilon} |^2\nonumber\\
        &\leq -q\int_{\Omega}v_{\varepsilon}^{-q+3}|\nabla v_{\varepsilon}|^{q-2}|D^2 \ln v_{\varepsilon}|^2.
    \end{align}
   By the Young inequality, there exist $c_2>0$ and $c_3>0$ such that
    \begin{align}\label{2.36}
        &-q\int_{\Omega}v_{\varepsilon}^{-q+1}|\nabla v_{\varepsilon}|^{q-2}\nabla v_{\varepsilon}\cdot\nabla (u_{\varepsilon}
        v_{\varepsilon})+(q-1)\int_{\Omega}u_{\varepsilon}v_{\varepsilon}^{-q+1}|\nabla v_{\varepsilon}|^{q}\nonumber\\
        &=-q\int_{\Omega}v_{\varepsilon}^{-q+2}|\nabla v_{\varepsilon}|^{q-2}\nabla v_{\varepsilon}\cdot\nabla u_{\varepsilon}
        -\int_{\Omega}u_{\varepsilon}v_{\varepsilon}^{-q+1}|\nabla v_{\varepsilon}|^{q}\nonumber\\
        &\leq q\int_{\Omega}v_{\varepsilon}^{-q+2}|\nabla v_{\varepsilon}|^{q-2}\nabla v_{\varepsilon}\cdot\nabla u_{\varepsilon}
        -\int_{\Omega}u_{\varepsilon}v_{\varepsilon}^{-q+1}|\nabla v_{\varepsilon}|^{q}\nonumber\\
        &\leq c_3\int_{\Omega}v_{\varepsilon}^{-q-1}|\nabla v_{\varepsilon}|^{q+2}+c_2\int_{\Omega}v_{\varepsilon}|\nabla
        u_{\varepsilon}|^{\frac{q+2}{3}} -\int_{\Omega}u_{\varepsilon}v_{\varepsilon}^{-q+1}|\nabla v_{\varepsilon}|^{q}.
    \end{align}
    On the other hand, since the identities $\nabla |\nabla
     v_{\varepsilon}|^2=2D^2v_{\varepsilon}\cdot\nabla v_{\varepsilon}$ and $|\Delta v_{\varepsilon}|=\sqrt{3}|D^2v_{\varepsilon}|$, it follows that
     for all $t \in(0, T_{\max, \varepsilon})$
     \begin{align}\label{2.37}
        &-q\int_{\Omega}v_{\varepsilon}^{-q+1}|\nabla v_{\varepsilon}|^{q-2}\nabla v_{\varepsilon}\cdot\nabla (u_{\varepsilon}
        v_{\varepsilon})+(q-1)\int_{\Omega}u_{\varepsilon}v_{\varepsilon}^{-q+1}|\nabla v_{\varepsilon}|^{q}\nonumber\\
        &=q(q-2)\int_{\Omega}u_{\varepsilon}v_{\varepsilon}^{-q+2}|\nabla v_{\varepsilon}|^{q-4}\nabla
        v_{\varepsilon}\cdot(D^2v_{\varepsilon}\cdot\nabla v_{\varepsilon})\nonumber\\
        &+q\int_{\Omega}u_{\varepsilon}v_{\varepsilon}^{-q+2}|\nabla v_{\varepsilon}|^{q-2}\Delta
        v_{\varepsilon}-(q-1)^2\int_{\Omega}u_{\varepsilon}v_{\epsilon}^{-q+1}|\nabla v_{\varepsilon}|^{q}\nonumber\\
        &\leq q(q-2+\sqrt{n})\int_{\Omega}u_{\varepsilon}v_{\varepsilon}^{-q+2}|\nabla v_{\varepsilon}|^{q-2}|D^2v_{\varepsilon}|
        -(q-1)^2\int_{\Omega}u_{\varepsilon}v_{\varepsilon}^{-q+1}|\nabla v_{\varepsilon}|^{q}\nonumber\\
        &\leq c_4\int_{\Omega}v_{\varepsilon}^{-q+1}|\nabla
        v_{\varepsilon}|^{q-2}|D^2v_{\varepsilon}|^2+c_5\int_{\Omega}u_{\varepsilon}^2v_{\varepsilon}^{-q+3}|\nabla
        v_{\varepsilon}|^{q-2}-(q-1)^2\int_{\Omega}u_{\varepsilon}v_{\varepsilon}^{-q+1}|\nabla v_{\varepsilon}|^{q}\nonumber\\
        &\leq c_4\int_{\Omega}v_{\varepsilon}^{-q+1}|\nabla
        v_{\varepsilon}|^{q-2}|D^2v_{\varepsilon}|^2+c_6\int_{\Omega}v_{\varepsilon}^{-q-1}|\nabla v_{\varepsilon}|^{q+2}\nonumber\\
        &+c_7\int_{\Omega}u_{\varepsilon}^{\frac{q+2}{2}}v_{\varepsilon}-(q-1)^2\int_{\Omega}u_{\varepsilon}v_{\varepsilon}^{-q+1}|\nabla
        v_{\varepsilon}|^{q},
    \end{align}
    where $c_i>0, i=4,\cdots, 7$,  are constants. Therefore  \eqref{2.29} follows from \eqref{2.31}--\eqref{2.36}, and
    \eqref{2.30} results from \eqref{2.31}--\eqref{2.35} and \eqref{2.37}.
\end{proof}

By means of the Gagliardo-Nirenberg inequality and the Ehrling lemma, we deduce the following functional inequality, which plays an important
role in estimating  an integral of the type $\int_{\Omega}\phi^{\beta}\psi$ below.
\begin{lemma}\label{lemma2.8}
Let $\Omega \subset \mathbb R^3$, and supposed that $p>0$ and $1< r<6$. Then for any $\eta>0$ one can find $C(\eta, p)>0$ such
that
\begin{align}\label{2.38}
    \|\phi^{\frac{p+1}{2}}\sqrt{\psi}\|^2_{L^{r}(\Omega)}\leq \eta \int_{\Omega}\phi^{p-1}\psi|\nabla \phi|^2+\eta
    \int_{\Omega}\phi^{p+1}\psi^{-1}|\nabla \psi|^2+C(\eta, p)\cdot
    \left\{\int_{\Omega}\phi\right\}^{p}\left\{\int_{\Omega}\phi\psi\right\}
\end{align}for all $\phi\in C^1(\Bar{\Omega})$ and $\psi\in C^1(\overline{\Omega})$  with $\phi>0$ and $\psi>0$ in $\overline{\Omega}$.
\end{lemma}
\begin{proof}
    From the Gagliardo-Nirenberg inequality, it follows that for $1<r<6$, there exist constants
    $a\in(0,1)$  and $c_1>0$ such that
    \begin{align*}
      \|\phi^{\frac{p+1}{2}}\sqrt{\psi}\|^2_{L^{r}(\Omega)}&\leq c_1\|\nabla
      (\phi^{\frac{p+1}{2}}\sqrt{\psi})\|^{2a}_{L^{2}(\Omega)}\|\phi^{\frac{p+1}{2}}\sqrt{\psi}\|^{2(1-a)}_{L^{1}(\Omega)}
      +c_1\|\phi^{\frac{p+1}{2}}\sqrt{\psi}\|_{L^{1}(\Omega)}^2.
          \end{align*}
    Furthermore,  by Young's inequality,  one can see that  for all $\eta >0$ there is $ c_2(\eta)>0$
    such that
    \begin{align}\label{2.39}
      c_1\|\nabla
      (\phi^{\frac{p+1}{2}}\sqrt{\psi})\|^{2a}_{L^{2}(\Omega)}\|\phi^{\frac{p+1}{2}}\sqrt{\psi}\|^{2(1-a)}_{L^{1}(\Omega)}&\leq \eta \|\nabla (\phi^{\frac{p+1}{2}}\sqrt{\psi})\|_{L^{2}(\Omega)}^2+c_2(\eta)\|\phi^{\frac{p+1}{2}}\sqrt{\psi}\|_{L^{1}(\Omega)}^2.
    \end{align}
     For $p\geq 1$, the compact embedding $L^{r}\hookrightarrow L^{1} $ and continuous embedding $L^{1}\hookrightarrow
    L^{\frac{2}{p+1}}$ allow us to apply Ehrling lemma, which yields
    $$c_2(\eta)\|\phi^{\frac{p+1}{2}}\sqrt{\psi}\|_{L^{1}(\Omega)}^2\leq
    \frac{1}{2}\|\phi^{\frac{p+1}{2}}\sqrt{\psi}\|^2_{L^{r}(\Omega)}+c_3(\eta)\|\phi^{\frac{p+1}{2}}\sqrt{\psi}\|^2_{L^{\frac 2{p+1}}(\Omega)}$$
   with some $c_3(\eta)>0$.  Therefore for all  $p\geq 1$, we have
    $$\|\phi^{\frac{p+1}{2}}\sqrt{\psi}\|^2_{L^{r}(\Omega)}\leq \eta \|\nabla
    (\phi^{\frac{p+1}{2}}\sqrt{\psi})\|_{L^{2}(\Omega)}^2+c_3(\eta)\|\phi^{\frac{p+1}{2}}\sqrt{\psi}\|^2_{L^{\frac{2}{p+1}}(\Omega)}.$$
   On the other hand, for $p\in(0,1)$, one can see  that
    \begin{align*}
      \|\phi^{\frac{p+1}{2}}\sqrt{\psi}\|^2_{L^{r}(\Omega)}
      &\leq \eta \|\nabla
      (\phi^{\frac{p+1}{2}}\sqrt{\psi})\|_{L^{2}(\Omega)}^2+c_2(\eta)\|\phi^{\frac{p+1}{2}}\sqrt{\psi}\|_{L^{1}(\Omega)}^2\nonumber\\
      &\leq \eta \|\nabla
      (\phi^{\frac{p+1}{2}}\sqrt{\psi})\|_{L^{2}(\Omega)}^2+c_4(\eta)\|\phi^{\frac{p+1}{2}}\sqrt{\psi}\|_{L^{\frac2{p+1}}(\Omega)}^2.
    \end{align*}
    Hence from  the H\"{o}lder inequality,  it follows that
    \begin{align}\label{2.40}
    &\eta \|\nabla
    (\phi^{\frac{p+1}{2}}\sqrt{\psi})\|_{L^{2}(\Omega)}^2+ (c_3(\eta)+ c_4(\eta))\|\phi^{\frac{p+1}2}\sqrt{\psi}\|_{L^{\frac2{p+1}}(\Omega)}^2\nonumber\\
    &\leq \eta \int_{\Omega}\phi^{p-1}\psi|\nabla \phi|^2+\eta \int_{\Omega}\phi^{p+1}\psi^{-1}|\nabla \psi|^2+C(\eta,p)\cdot
    \left\{\int_{\Omega}\phi\right\}^{p}\cdot\left\{\int_{\Omega}\phi\psi\right\}.
\end{align}
\end{proof}

\section{Estimates of $ \|u_{\epsilon}(\cdot,t)\|_{L^p(\Omega)}$ }\label{sec:3}
\subsection{Space-time estimate of $u_{\varepsilon}^{p_*-1}v_{\varepsilon}|\nabla u_{\varepsilon}|^2$}

The purpose of  this  subsection is to derive the
space-time estimate of $u_{\varepsilon}^{p_*}v_{\varepsilon}|\nabla u_{\varepsilon}|^2$ with $p^*\in [1-\alpha,0]$. This refines the space-time estimate of $u_{\varepsilon}^{1-\alpha}v_{\varepsilon}|\nabla u_{\varepsilon}|^2$ provided in Lemma \ref{lemma2.3},  and allows us to obtain the bounds for
$\|u_{\varepsilon}(\cdot,t)\|_{L^{p_0}(\Omega)}$ with $p_0>\frac{3}{2}$ in the subsequent subsection. To this end,
as outlined in the introduction,  we shall make sure that $\mathcal{ G_{\varepsilon}}(t)$ enjoys the feature of the  energy-like  quantity.
\begin{lemma}\label{lemma3.1}
Let $\Omega \subset \mathbb R^3$. Then one can find constants $a>0$ and $b>0$ such that for all $t\in (0,T_{max,\varepsilon})$ the functional
$$\mathcal{ G_{\varepsilon}} (t):=\int_{\Omega}u_{\varepsilon}(\cdot,t)\ln{u_{\varepsilon}(\cdot,t)}+a\int_{\Omega}\frac{|\nabla
v_{\varepsilon}(\cdot,t)|^4}{v_{\varepsilon}^{3}(\cdot,t)}$$ satisfies
\begin{align}\label{3.1}
 \mathcal{ G_{\varepsilon}}'(t)&+b\int_{\Omega}\frac{|\nabla
   v_{\varepsilon}|^{2}}{v_{\varepsilon}}|D^2\ln{v_{\varepsilon}}|^2+b\int_{\Omega}u_{\varepsilon}\frac{|\nabla
   v_{\varepsilon}|^4}{v_{\varepsilon}^3}+\frac{1}{2}\int_{\Omega}v_{\varepsilon}|\nabla u_{\varepsilon}|^2 \nonumber
    \\
    &\leq \chi^2\int_{\Omega}u_{\varepsilon}^{2\alpha-2}v_{\varepsilon}|\nabla v_{\varepsilon}|^2+\ell\int_{\Omega} u_{\varepsilon}
    v_{\varepsilon} \ln u_{\varepsilon}+\ell\int_{\Omega}u_{\varepsilon}v_{\varepsilon}+\frac{1}{4}\int_{\Omega}v_{\varepsilon}.
\end{align}
\end{lemma}
\begin{proof}
   Multiplying the first equation in \eqref{2.5} by $1+\ln u_{\varepsilon}$, we use the Young inequality and $\ln \xi \leqslant \xi$ for all $\xi>0$ to infer
   that
\begin{align}\label{3.2}
& \frac{d}{dt} \int_{\Omega} u_{\varepsilon} \ln u_{\varepsilon}+\int_{\Omega} v_{\varepsilon}\left|\nabla u_{\varepsilon}\right|^2\nonumber \\
=& \chi \int_{\Omega} u_{\varepsilon}^{\alpha-1} v_{\varepsilon} \nabla u_{\varepsilon} \cdot \nabla v_{\varepsilon}+ \ell\int_{\Omega}
u_{\varepsilon} v_{\varepsilon} \ln u_{\varepsilon} + \ell\int_{\Omega} u_{\varepsilon} v_{\varepsilon}\nonumber\\
\leq &  \frac{1}{4} \int_{\Omega} v_{\varepsilon}\left|\nabla u_{\varepsilon}\right|^2+ \chi^2 \int_{\Omega} u_{\varepsilon}^{2\alpha-2}
v_{\varepsilon}\left|\nabla v_{\varepsilon}\right|^2 +\ell\int_{\Omega} u_{\varepsilon} v_{\varepsilon} \ln u_{\varepsilon}+\ell\int_{\Omega}
u_{\varepsilon} v_{\varepsilon}
\end{align}
for all $t \in(0, T_{max, \varepsilon})$. On the other hand, applying \eqref{2.30} to $q:=4$, we have
\begin{align}\label{3.3}
    \frac{d}{dt}\int_{\Omega}\frac{|\nabla v_{\varepsilon}|^4}{v_{\varepsilon}^{3}}+\gamma(4)\int_{\Omega}\frac{|\nabla
    v_{\varepsilon}|^{2}}{v_{\varepsilon}}|D^2\ln{v_{\varepsilon}}|^2+\gamma(4)\int_{\Omega}u_{\varepsilon}\frac{|\nabla
    v_{\varepsilon}|^4}{v_{\varepsilon}^{3}}\leq \frac{1}{\gamma(4)}\int_{\Omega}v_{\varepsilon}(|\nabla u_{\varepsilon}|^{2}+1).
\end{align}
Combining \eqref{3.2} with \eqref{3.3}, we arrive at
\begin{align*}
&\frac{d}{dt}\left\{\int_{\Omega} u_{\varepsilon} \ln u_{\varepsilon}+\frac{\gamma(4)}{4}\int_{\Omega} \frac{\left|\nabla
v_{\varepsilon}\right|^4}{v_{\varepsilon}^3}\right\}\nonumber
\\
&+\frac{\gamma^2(4)}{4}\int_{\Omega}\frac{|\nabla
v_{\varepsilon}|^{2}}{v_{\varepsilon}}|D^2\ln{v_{\varepsilon}}|^2+\frac{\gamma^2(4)}{4}\int_{\Omega}u_{\varepsilon}\frac{|\nabla
v_{\varepsilon}|^4}{v_{\varepsilon}^3}+\frac{1}{2}\int_{\Omega}v_{\varepsilon}|\nabla u_{\varepsilon}|^2 \nonumber
    \\
 &\leq \chi^2\int_{\Omega}u_{\varepsilon}^{2\alpha-2}v_{\varepsilon}|\nabla v_{\varepsilon}|^2+\ell\int_{\Omega} u_{\varepsilon} v_{\varepsilon}
 \ln u_{\varepsilon}+\ell\int_{\Omega}u_{\varepsilon}v_{\varepsilon}+\frac{1}{4}\int_{\Omega}v_{\varepsilon}
\end{align*}
for all $t \in(0, T_{max, \varepsilon})$. This readily leads to \eqref{3.1} with $a=\frac{\gamma(4)}{4}$ and
 $b=\frac{\gamma^2(4)}{4}$.
\end{proof}

By Lemma \ref{lemma2.3}, Lemma \ref{lemma2.5} and Lemma \ref{lemma2.6}, we can obtain the estimate of $\int_0^{T_{\max}}\int_\Omega u_{\varepsilon}^{\frac{5}{3}}v_{\varepsilon}$, which allows us to address the unfavorable contributions on the right-hand side of \eqref{3.1} whenever the exponent $\alpha\in\left(\frac{3}{2},2\right]$.
\begin{lemma} \label{lemma3.2}
Let $\Omega \subset \mathbb R^3$, suppose that $\alpha\in\left(\frac{3}{2},2\right]$ and $K>0$ with the
property \eqref{1.6} be valid. Then there exists  $C=C(K)>0$ such that
\begin{align}\label{3.4}
    \int_0^{T_{max,\varepsilon}}\int_{\Omega}u_{\varepsilon}^{\frac{5}{3}}v_{\varepsilon}\leq C.
\end{align}
\end{lemma}
\begin{proof}
    Since \eqref{2.9} and \eqref{1.6}, there exists $c_1=c_1(K)>0$ such that
    \begin{align*}
        \int_{\Omega}u_{\varepsilon}(\cdot,t)\leq c_1 \qquad for\; all\; t\in(0,T_{max,\varepsilon}).
    \end{align*}
  Due to $\alpha\in(\frac{3}{2},2]$, then $q:=2-\alpha\in[0,\frac{2}{3}]$. As an application of Lemma \ref{lemma2.6} to $p^*:=1$,  there exists $c_2>0$
   such that
   \begin{align*}
    \int_0^{t}\int_{\Omega}u_{\varepsilon}^{\frac{5}{3}}v_{\varepsilon}\leq c_2
    \left\{\int_0^{t}\int_{\Omega}u_{\varepsilon}^{1-\alpha}v_{\varepsilon}|\nabla
    u_{\varepsilon}|^2+\int_0^{t}\int_{\Omega}\frac{u_{\varepsilon}}{v_{\varepsilon}}|\nabla
    v_{\varepsilon}|^2+\int_0^{t}\int_{\Omega}u_{\varepsilon}v_{\varepsilon}\right\}
\end{align*}
for all $t \in(0, T_{max, \varepsilon})$. Hence from \eqref{2.10}, \eqref{2.13} and \eqref{2.19}, there exists  $C(K)>0$ such that \eqref{3.4} is valid.
\end{proof}

With the help of  two lemmas above, the quasi-energy feature of $G_{\varepsilon}(t)$ can be verified and thereby  space-time estimates of  $u_{\varepsilon}\frac{|\nabla v_{\varepsilon}|^4}{v_{\varepsilon}^{3}} $ and $v_{\varepsilon}|\nabla u_{\varepsilon}|^2$ are established.
\begin{lemma} \label{lemma3.3}
Let $\Omega \subset \mathbb R^3$ and $\alpha\in\left(\frac{3}{2},\frac{5}{3}\right]$, suppose that $K>0$ with the
property \eqref{1.6} be valid. Then there exists $C=C(K)>0$ such that
\begin{align}\label{3.5}
    \int_0^{T_{max,\varepsilon}}\int_{\Omega}u_{\varepsilon}\frac{|\nabla v_{\varepsilon}|^4}{v_{\varepsilon}^{3}}\leq C
\end{align}
and
\begin{align}\label{3.6}
    \int_0^{T_{max,\varepsilon}}\int_{\Omega}v_{\varepsilon}|\nabla u_{\varepsilon}|^2\leq C.
\end{align}
\end{lemma}
\begin{proof}
   In light of condition $\alpha\in\left(\frac{3}{2},\frac{5}{3}\right]$, $4\alpha-5\in(0,\frac{5}{3}]$. Therefore from  Lemma
   \ref{lemma3.1}, it follows that for all $t\in (0,T_{max,\varepsilon})$,
    \begin{align}\label{3.7}
  \mathcal{ G_{\varepsilon}}'(t)&+b\int_{\Omega}\frac{|\nabla
   v_{\varepsilon}|^{2}}{v_{\varepsilon}}|D^2\ln{v_{\varepsilon}}|^2+b\int_{\Omega}u_{\varepsilon}\frac{|\nabla
   v_{\varepsilon}|^4}{v_{\varepsilon}^3}+\frac{1}{2}\int_{\Omega}v_{\varepsilon}|\nabla u_{\varepsilon}|^2 \nonumber
    \\
    &\leq \chi^2\int_{\Omega}u_{\varepsilon}^{2\alpha-2}v_{\varepsilon}|\nabla v_{\varepsilon}|^2+\ell\int_{\Omega} u_{\varepsilon}
    v_{\varepsilon} \ln u_{\varepsilon}+\ell\int_{\Omega}u_{\varepsilon}v_{\varepsilon}+\frac{1}{4}\int_{\Omega}v_{\varepsilon}\nonumber
   \\
    &\leq c_1\int_{\Omega}u_{\varepsilon}^{4\alpha-5}v_{\varepsilon}+\frac{b}{2}\int_{\Omega}u_{\varepsilon}\frac{|\nabla v_{\varepsilon}|^4}{v_{\varepsilon}^3}+\ell\int_{\Omega}
    u_{\varepsilon}^{\frac{5}{3}}v_{\varepsilon}+\ell\int_{\Omega}u_{\varepsilon}v_{\varepsilon}+\frac{1}{4}\int_{\Omega}v_{\varepsilon}\nonumber
    \\
    &\leq (c_1+ \ell) \int_{\Omega} u_{\varepsilon}^{\frac{5}{3}}v_{\varepsilon}+\frac{b}{2}\int_{\Omega}u_{\varepsilon}\frac{|\nabla
    v_{\varepsilon}|^4}{v_{\varepsilon}^3}+\ell\int_{\Omega}u_{\varepsilon}v_{\varepsilon}+(c_1+1)\int_{\Omega}v_{\varepsilon}
\end{align}
with  $c_1:=\frac{\chi^2\|v_0\|_{L^{\infty}}^4}{b}$ and hence
 \begin{align}\label{3.8}
    &\mathcal{ G_{\varepsilon}}'(t)+b\int_{\Omega}\frac{|\nabla
    v_{\varepsilon}|^{2}}{v_{\varepsilon}}|D^2\ln{v_{\varepsilon}}|^2+\frac{b}{2}\int_{\Omega}u_{\varepsilon}\frac{|\nabla
    v_{\varepsilon}|^4}{v_{\varepsilon}^3}+\frac{1}{2}\int_{\Omega}v_{\varepsilon}|\nabla u_{\varepsilon}|^2\nonumber
    \\
    &\leq c_2\int_{\Omega}
    u_{\varepsilon}^{\frac{5}{3}}v_{\varepsilon}+
    \ell\int_{\Omega}u_{\varepsilon}v_{\varepsilon}+c_2\int_{\Omega}v_{\varepsilon}
\end{align}
with $c_2:=c_1+1+\ell$. After an integration over time, this together with \eqref{3.4}, \eqref{2.10}, \eqref{2.18} and \eqref{1.6} entails that
\begin{align*}
     & \frac{b}{2}\int_0^{t}\int_{\Omega}u_{\varepsilon}\frac{|\nabla
     v_{\varepsilon}|^4}{v_{\varepsilon}^3}+\frac{1}{2}\int_0^{t}\int_{\Omega}v_{\varepsilon}|\nabla
     u_{\varepsilon}|^2+b\int_0^{t}\int_{\Omega}\frac{|\nabla v_{\varepsilon}|^{2}}{v_{\varepsilon}}|D^2\ln{v_{\varepsilon}}|^2 \nonumber
    \\
    \leq& c_2\int_0^{t}\int_{\Omega}
    u_{\varepsilon}^{\frac{5}{3}}v_{\varepsilon}+\ell\int_0^{t}\int_{\Omega}u_{\varepsilon}v_{\varepsilon}+c_2\int_0^{t}\int_{\Omega}v_{\varepsilon}\nonumber
    \\
    &+\int_{\Omega} u_{0\varepsilon} \ln u_{0\varepsilon}+a \int_{\Omega} \frac{\left|\nabla
    v_{0\varepsilon}\right|^4}{v_{0\varepsilon}^3}-\int_{\Omega} u_{\varepsilon} \ln u_{\varepsilon}\nonumber
    \\
    \leq &c_2\int_0^{t}\int_{\Omega}
    u_{\varepsilon}^{\frac{5}{3}}v_{\varepsilon}+\ell\int_0^{t}\int_{\Omega}u_{\varepsilon}v_{\varepsilon}+c_2
\int_0^{t}\int_{\Omega}v_{\varepsilon}\nonumber
    \\
    &+\int_{\Omega} (u_0+1)^2+a\int_{\Omega} \frac{\left|\nabla
    v_{0}\right|^4}{v_{0}^3}+\frac{|\Omega|}{e}\nonumber
    \\
    \leq&C(K)
\end{align*}
for all $t\in (0,T_{max,\varepsilon})$, and thereby completes the proof.
\end{proof}

As a direct consequence of  \eqref{3.6} and  \eqref{2.13}, we have

\begin{lemma} \label{lemma3.4} Let $\alpha\in\left(\frac{3}{2},\frac{5}{3}\right]$ and $K>0$ with the
property \eqref{1.6} be valid. Then for any  $p_*\in [1-\alpha, 0]$ and $C=C(K)>0$ such that
\begin{align}\label{3.9}
\int_0^{T_{max,\varepsilon}}\int_{\Omega}u_{\varepsilon}^{p_*}v_{\varepsilon}|\nabla u_{\varepsilon}|^2\leq C.
\end{align}
\end{lemma}
\begin{proof} Since  $ u_{\varepsilon}^{p_*}\leq 1+u_{\varepsilon}^{1-\alpha}$
for   any
$p_*\in [1-\alpha, 0]$, we have
$$
\int_0^{T_{max,\varepsilon}}\int_{\Omega} u_{\varepsilon}^{p_*}v_{\varepsilon}|\nabla u_{\varepsilon}|^2 \leq  \int_0^{T_{max,\varepsilon}}\int_{\Omega}
v_{\varepsilon}|\nabla u_{\varepsilon}|^2+\int_0^{T_{max,\varepsilon}}\int_{\Omega}u_{\varepsilon}^{1-\alpha}v_{\varepsilon}|\nabla u_{\varepsilon}|^2,
$$
which along with  \eqref{3.6} and \eqref{2.13} entails \eqref{3.9}   readily.
 \end{proof}

\subsection{ $L^{p_0}$-estimates of $u_{\varepsilon}$ with some $p_0>\frac{3}{2}$}

This subsection establishes the $L^{p_0}$-bounds for $u_{\varepsilon}$ with some $p_0>\frac{3}{2}$, which serves as the foundation for a bootstrap argument to achieve  the $L^{p}$ bounds of  $u_{\varepsilon}$ for arbitrary $p>1$.
As a prerequisite for the derivation of estimate on $L^{p_0}$-bounds for $u_{\varepsilon}$, we apply  Lemma \ref{lemma2.8} to show
that  the integral of the form
$\int_{\Omega}\varphi^{\beta}\psi$  can be dominated by
$\int_{\Omega}\varphi^{k}\psi|\nabla \varphi|^2$ when added $
    \int_{\Omega}\varphi\frac{|\nabla \psi|^4}{\psi^3}
$ 
and term $\int_{\Omega}  \varphi \psi$.
\begin{lemma}\label{lemma3.5} Let $L>0$, $k\in(-1,-\frac{1}{3})$ and $\beta\in[1,k+\frac{8}{3})$. There exists $C=C(L,k, \beta)>0$ such that
if  $\int_{\Omega}\varphi\leq L$, then
  \begin{align}\label{3.10}
    \int_{\Omega}\varphi^{\beta}\psi\leq  \int_{\Omega}\varphi^{k}\psi|\nabla \varphi|^2+\int_{\Omega}\varphi\frac{|\nabla \psi|^4}{\psi^3}+C\int_{\Omega}\varphi \psi
\end{align} is valid for all $\varphi\in C^1(\overline \Omega)$ and $\psi \in C^1(\overline \Omega)$ fulfilling $\varphi> 0$ and $\psi>0$ in $\overline \Omega$.
\end{lemma}
\begin{proof}
    Let
 $\theta:=\frac{1}{k+3-\beta}$ and $ \theta_*:=\frac{1}{\beta-k-2}.$
   Then  \begin{align}\label{3.11}
        1\leq\theta<3
    \end{align}
   can be warranted by $\beta \in [k+2,k+\frac{8}{3})$.
Consequently, by the H\"{o}lder inequality and applying  Lemma \ref{lemma2.8} to $\eta=\frac12$, we conclude that
there exists
$c_1=c_1(\beta,k,L)$ such that for all $t\in (0,T_{max,\varepsilon})$,
\begin{align}\label{3.12}
    \int_{\Omega}\varphi^{\beta}\psi
    &=\int_{\Omega}(\varphi^{\frac{k+2}{2}}\psi^{\frac{1}{2}})^2\varphi^{\beta-k-2} \nonumber
    \\
    &\leq\|\varphi^{\frac{k+2}{2}}\psi^{\frac{1}{2}}\|^2_{L^{2\theta}(\Omega)}\cdot\|\varphi^{\beta-k-2}\|_{L^{\theta_*}(\Omega)}\nonumber
    \\
    &\leq L^{\frac{1}{\theta_*}}\|\varphi^{\frac{k+2}{2}}\psi^{\frac{1}{2}}\|^2_{L^{2\theta}(\Omega)}\nonumber
    \\
    &\leq \frac{1}{2}\int_{\Omega}\varphi^{k}\psi|\nabla
    \varphi|^2+\frac{1}{2}
    \int_{\Omega}\varphi^{k+2}\frac{|\nabla
    \psi|^2}{\psi}+c_1\int_{\Omega}\varphi \psi.
\end{align}
Furthermore, let
$\lambda:=-\frac{1}{k}$ and $ \lambda_*:=\frac{1}{k+1}$. Then the restriction  $k\in(-1,-\frac{1}{3})$ warrants that
$     1<\lambda<3$.
 Applying Lemma \ref{lemma2.8} to $p:=k+1, r:=2\lambda, \eta:=\min\{L^{-\frac1{\lambda_*}},1\}$ and Young's inequality once more, we obtain that there
exists positive constant $c_2=c_2(k,\beta,L)$ such that for all $t\in (0,T_{max,\varepsilon})$,
\begin{align}\label{3.13}
    \int_{\Omega}\varphi^{k+2}\frac{|\nabla \psi|^2}{\psi}&\leq
    \frac{1}{2}\int_{\Omega}\varphi\frac{|\nabla
    \psi|^4}{\psi^3}+\frac1{2}\int_{\Omega}\varphi^{2k+3}\psi\nonumber
    \\
    &=\frac{1}{2}\int_{\Omega}\varphi\frac{|\nabla
    \psi|^4}{\psi^3}+\frac1{2}\int_{\Omega}(\varphi^{\frac{k+2}{2}}\psi^{\frac{1}{2}})^2\varphi^{k+1}\nonumber
    \\
    &\leq\frac{1}{2}\int_{\Omega}\varphi\frac{|\nabla
    \psi|^4}{\psi^3}+\frac1{2}\|\varphi^{\frac{k+2}{2}}\psi^{\frac{1}{2}}\|_{L^{2\lambda}(\Omega)}^2\|\varphi^{k+1}\|_{L^{\lambda_*}(\Omega)}\nonumber
    \\
    &\leq\frac{1}{2}\int_{\Omega}\varphi\frac{|\nabla
    \psi|^4}{\psi^3}+\frac{L^{\frac 1{ \lambda_* }}}{2}\|\varphi^{\frac{k+2}{2}}\psi^{\frac{1}{2}}\|_{L^{2\lambda}(\Omega)}^2\nonumber
    \\
     &\leq\frac{1}{2}\int_{\Omega}\varphi\frac{|\nabla
     \psi|^4}{\psi^3}+\frac{1}{2}\int_{\Omega}\varphi^{k+2}\frac{|\nabla
     \psi|^2}{\psi}+\frac{1}{2}\int_{\Omega}\varphi^{k}\psi|\nabla
    \varphi|^2+c_2\int_{\Omega}\varphi\psi.
\end{align}
 Therefore, if $\beta \in [k+2,k+\frac{8}{3})$, it follows from \eqref{3.12} and \eqref{3.13}  that
\begin{align} \label{3.14}
\int_{\Omega}\varphi^{\beta}\psi\leq \eta \int_{\Omega}\varphi^{k}\psi|\nabla \varphi|^2+
\int_{\Omega}\varphi\frac{|\nabla \psi|^4}{\psi^3}+C(\beta, k, L)\int_{\Omega}\varphi\psi
\end{align}
for all $t\in (0,T_{max,\varepsilon})$.
 Apart from that, it is observed that for $\beta \in [1,k+2)$,
$$ \int_{\Omega}\varphi^{\beta}\psi\leq  \int_{\Omega}\varphi^{k+2}\psi+
\int_{\Omega}\varphi\psi\quad \hbox{for all} \,\, t\in (0,T_{max,\varepsilon}).
$$
This together with \eqref{3.14} implies \eqref{3.10} readily.
\end{proof}

The combination of Lemmas \ref{lemma3.3}, Lemma \ref{lemma3.4} and Lemma \ref{lemma3.5} yields the result of follows.
\begin{lemma}\label{lemma3.6}
Assume $\alpha\in\left(\frac{3}{2},\frac{19}{12}\right)$ and
$K>0$ be parameter such that property \eqref{1.6} is valid. Then there exist $p_0>\frac{3}{2}$ and $C=C(K)>0$ such that
\begin{align}\label{3.15}
    \int_{\Omega}u_{\varepsilon}^{p_0}(\cdot,t)\leq C\qquad for\;all\; t\in (0,T_{max,\varepsilon}).
\end{align}

\end{lemma}
\begin{proof}According to \eqref{2.9}, there exists $c_1=c_1(K)>0$ such that
    \begin{align*}
        \int_{\Omega}u_{\varepsilon}(\cdot,t)\leq c_1\qquad for\; all\; t\in (0,T_{max,\varepsilon}).
    \end{align*}

Let $ \delta:=\min\{\frac 12 (\alpha-\frac{4}{3}) ,\frac{19}{12}-\alpha,\frac12 \}>0$ and
$ p_0:=\frac{3}{2}+  \frac{\delta}2$.
Then multiplying the first equation in \eqref{2.5} by $u_{\varepsilon}^{p_0-1}$ and by the Young inequality, we infer that
\begin{align}\label{3.17}
      & \frac{1}{p_0}\frac{d}{dt}\int_{\Omega}u_{\varepsilon}^{p_0}+\frac{p_0-1}{2}\int_{\Omega}u_{\varepsilon}^{p_0-1}v_{\varepsilon}|\nabla
    u_{\varepsilon}|^2\nonumber\\
    &\leq \frac{(1-p_0)\chi^2}{2}\int_{\Omega}u_{\varepsilon}^{p_0+2\alpha-3}v_{\varepsilon}|\nabla
    v_{\varepsilon}|^2+\ell\int_{\Omega}u_{\varepsilon}^{p_0}v_{\varepsilon}\nonumber
    \\
    &\leq \frac{(1-p_0)\chi^2\|v_0\|_{L^{\infty}}^4}{2}\int_{\Omega}u_{\varepsilon}^{2p_0+4\alpha-7}v_{\varepsilon}+
    \int_{\Omega}u_{\varepsilon}\frac{|\nabla
    v_{\varepsilon}|^4}{v_{\varepsilon}^3}+\ell\int_{\Omega}u_{\varepsilon}^{\frac 53}v_{\varepsilon}
    +\ell\int_{\Omega} u_{\varepsilon}v_{\varepsilon}.
\end{align}

 Let $k:=-\frac13-\delta $. Then $k\in (-1,-\frac13)$ due to the fact $\delta<\frac 23$. In addition,  by the definition of $\delta$, we can see that $k\in (1-\alpha,0)$. Therefore  $k\in (1-\alpha,0)\bigcap (-1,-\frac13)$,  and thus
 \begin{align}\label{3.18}
\int_0^{T_{max,\varepsilon}}\int_{\Omega}u_{\varepsilon}^{k}v_{\varepsilon}|\nabla u_{\varepsilon}|^2\leq c_1(K,T_{max,\varepsilon})
\end{align}
for some  $c_1(K,T_{max,\varepsilon})>0$ by   Lemma \ref{lemma3.4}.

 On other hand, according to the definitions of $\delta$ and $p_0$, we have  $\delta<\frac{19}{12}-\alpha$ and thereby  $2p_0+4\alpha-7< \frac 73-\delta =k+\frac 83$.
 Apart from that, it is claimed that
  \begin{align}\label{3.19}
  2p_0+4\alpha-7>1.
  \end{align}
  Indeed, according to the definition of  $p_0$,  it is easy to see that
  $ 2p_0+4\alpha-8=\delta+4\alpha -5>0$  due to $\alpha>\frac 32$.

 Therefore applying Lemma \ref{lemma3.5} to $\beta:=2p_0+4\alpha-7$,   we can conclude that
   there exists $c_2 >0$ such that
\begin{align}\label{3.21}
   \frac{(1-p_0)\chi^2\|v_0\|_{L^{\infty}}^4}{2}\int_{\Omega}u_{\varepsilon}^{2p_0+4\alpha-7}v_{\varepsilon}\leq
    \int_{\Omega}u_{\varepsilon}^{k}v_{\varepsilon}|\nabla u_{\varepsilon}|^2+\int_{\Omega}u_{\varepsilon}\frac{|\nabla
    v_{\varepsilon}|^4}{v_{\varepsilon}^3}+c_2\int_{\Omega}u_{\varepsilon}v_{\varepsilon}
\end{align}
for  $t\in (0,T_{max,\varepsilon})$. Combining \eqref{3.21}, \eqref{3.17}  with  \eqref{3.18} and  integrating the results over time, together with  \eqref{3.5}, \eqref{1.6}, \eqref{2.18} and \eqref{2.10},  we have
\begin{align}\label{3.22}
     & \frac{1}{p_0}\int_{\Omega}u_{\varepsilon}^{p_0} +\frac{p_0-1}{2}\int_0^{t}\int_{\Omega}u_{\varepsilon}^{p_0-1}v_{\varepsilon}|\nabla
     u_{\varepsilon}|^2 \nonumber
    \\
    &\leq \int_0^{t}\int_{\Omega}u_{\varepsilon}^{k}v_{\varepsilon}|\nabla
    u_{\varepsilon}|^2+2\int_0^{t}\int_{\Omega}u_{\varepsilon}\frac{|\nabla
    v_{\varepsilon}|^4}{v_{\varepsilon}^3}+c_2\int_0^{t}\int_{\Omega}u_{\varepsilon}v_{\varepsilon}+
    \int_{\Omega}(u_{0}+1)^{p_0}+\int_0^{t}\int_{\Omega}v_{\varepsilon}\nonumber
    \\
    &\leq C(K,T_{max,\varepsilon})
\end{align}
for all $t\in (0,T_{max,\varepsilon})$, and thus complete the proof.
 \end{proof}

\subsection{ $L^{p}$-estimates of $u_{\varepsilon}$ for all $p> 1$}

The aim of this subsection is to achieve  $L^p$-bounds of $u_{\varepsilon}$ for all $p> 1$.  Of crucial importance to our approach in this direction is to refine \eqref{3.10} in Lemma \ref{lemma3.5}, which involves the term $\int_{\Omega}\varphi\frac{|\nabla
    \psi|^q}{\psi^{q-1}}$ instead of $\int_{\Omega}\varphi\frac{|\nabla
    \psi|^4}{\psi^{3}}$,  when the enhanced integrability properties of  $\varphi$ is taken into account.
\begin{lemma}\label{lemma3.7}
For $L>0,p_0>\frac{3}{2}$, $p>1$, $q>2+\frac{3p}{p_0}$, $\beta\in [1,\frac{2p_0}{3}+p+1)$
and any $\eta\in(0,1)$, there exists $C=C(\eta, L,p,\beta
 )>0$ such that
whenever  $\int_{\Omega}\varphi^{p_0}\leq L$,
  \begin{align}\label{3.25}
    \int_{\Omega}\varphi^{\beta} \psi\leq \eta \int_{\Omega}\varphi^{p-1}\psi|\nabla
    \varphi|^2+\eta \int_{\Omega}\varphi\frac{|\nabla  \psi|^q}{{\psi}^{q-1}}+C\int_{\Omega}\varphi \psi
\end{align} holds for   all $\varphi\in C^1(\overline \Omega)$ and $\psi \in C^1(\overline \Omega)$ fulfilling $\varphi> 0$ and $\psi>0$ in $\overline \Omega$.
\end{lemma}
\begin{proof}
    It is observed that the restriction $q>2+\frac{3p}{p_0}$ ensures that
    $ \frac{qp+q-2}{q-2}<\frac{2p_0}{3}+p+1.$
    Now we firstly verify \eqref{3.25} in the case
        \begin{align}\label{3.26}
        \beta \in \left[\frac{qp+q-2}{q-2},\frac{2p_0}{3}+p+1\right).
    \end{align}
    To this end, let
 $$\rho:=\frac{p_0}{p_0-\beta+p+1}\quad and \quad \rho_*:=\frac{p_0}{\beta-p-1}.$$
    Due to \eqref{3.26}, it follows that
    \begin{align}\label{3.28}
        \rho\geq1\quad and\quad 1<2\rho<6.
    \end{align}
Hence, thanks to  $\int_{\Omega}\varphi^{p_0}\leq L$ for $p_0>\frac{3}{2}$, the H\"{o}lder inequality, the Young inequality and Lemma \ref{lemma2.8},  we deduce that for any $\eta>0$, there exists
positive constant  $c_1=c_1(\eta,p,\beta,L)$  such that
\begin{align*}
    \int_{\Omega}\varphi^{\beta}\psi&=\int_{\Omega}(\varphi^{\frac{p+1}{2}}\psi^{\frac{1}{2}})^2\varphi^{\beta-p-1}\nonumber
    \\
    &\leq\|\varphi^{\frac{p+1}{2}}\psi^{\frac{1}{2}}\|^2_{L^{2\rho}(\Omega)}\cdot\|\varphi^{\beta-p-1}\|_{L^{\rho_*}(\Omega)}\nonumber
    \\
    &=\|\varphi^{\frac{p+1}{2}}\psi^{\frac{1}{2}}\|^2_{L^{2\rho}(\Omega)}\cdot\|\varphi\|_{L^{p_0}(\Omega)}^{\beta-p-1}\nonumber
    \\
    &\leq L^{\frac {\beta-p-1}{p_0}}\|\varphi^{\frac{p+1}{2}}\psi^{\frac{1}{2}}\|^2_{L^{2\rho}(\Omega)}\nonumber
    \\
    &\leq \frac{\eta}{2}\int_{\Omega}\varphi^{p-1}\psi|\nabla
    \varphi|^2+\frac{\eta}{2}\int_{\Omega}\varphi^{p+1}\frac{|\nabla
    \psi|^2}{\psi}+c_1\int_{\Omega}\varphi\psi\nonumber
    \\
    &\leq \frac{\eta}{2}\int_{\Omega}\varphi^{p-1}\psi|\nabla
    \varphi|^2+\frac{\eta}{2}\int_{\Omega}\varphi\frac{|\nabla
    \psi|^q}{\psi^{q-1}}+\frac 12\int_{\Omega}\varphi^{\frac{qp+q-2}{q-2}}\psi+c_1\int_{\Omega}\varphi\psi\nonumber
     \\
    &\leq \frac{\eta}{2}\int_{\Omega}\varphi^{p-1}\psi|\nabla
    \varphi|^2+\frac{\eta}{2}\int_{\Omega}\varphi\frac{|\nabla
    \psi|^q}{\psi^{q-1}}+
    \frac{1}{2}\int_{\Omega}\varphi^{\beta}\psi+
    (c_1+1)\int_{\Omega}\varphi\psi.
\end{align*}
 and thus
\begin{align}\label{3.29}
 \int_{\Omega}\varphi^{\beta}\psi\leq \eta\int_{\Omega}\varphi^{p-1}\psi|\nabla
\varphi|^2+\eta\int_{\Omega}\varphi\frac{|\nabla
\psi|^q}{\psi^{q-1}}+2 (c_1+1) \int_{\Omega}\varphi\psi.
\end{align}
Furthermore, it is observed that  for $\beta \in [1, \frac{qp+q-2}{q-2})$, we have $$\int_{\Omega}\varphi^{\beta}\psi\leq
    \int_{\Omega}\varphi^{\frac{qp+q-2}{q-2}}\psi+\int_{\Omega}\varphi\psi, $$
which along with \eqref{3.29}  complete the proof readily.
\end{proof}

We are now  in the position to achieve $L^p$-estimates for $u_{\varepsilon}$ for any $p>1$
 through tracing the evolution of $ \mathcal{ H}(t):=
 \int_{\Omega}u_{\varepsilon}^{p}+\int_{\Omega}\frac{|\nabla
   v_{\varepsilon}|^q}{v_{\varepsilon}^{q-1}}$,  thanks to the $L^{p_0}$-estimates of $u_{\varepsilon}$ with some $p_0>\frac{3}{2}$ established in Lemma \ref{lemma3.6}.
\begin{lemma}\label{lemma3.8}
Let $\Omega \subset \mathbb R^3$ and $K>0$ with the property that \eqref{1.6} holds. Supposed that
$\alpha\in\left(\frac{3}{2},\frac{19}{12}\right)$  and $p>1$, there exists  $C=C( K, p)>0$ such that for all
 $ t\in (0,T_{max,\varepsilon})$,
\begin{align}\label{3.30}
   \int_{\Omega}u_{\varepsilon}^{p}\leq C\quad \hbox{and}\quad
   \int_0^{T_{max,\varepsilon}}\int_{\Omega}u_{\varepsilon}^{p-1}v_{\varepsilon}|\nabla u_{\varepsilon}|^2\leq C
\end{align}
as well as
\begin{align}\label{3.30-1}
   \int_0^{T_{max,\varepsilon}}\int_{\Omega}u_{\varepsilon}^{p+1}v_{\varepsilon}\leq C.
\end{align}
\end{lemma}
\begin{proof}
From \eqref{2.29} in Lemma \ref{lemma2.7}, it follows that
    \begin{align}\label{3.31}
    \frac{d}{dt}\int_{\Omega}\frac{|\nabla v_{\varepsilon}|^q}{v_{\varepsilon}^{q-1}}+\Gamma(q)\int_{\Omega}\frac{|\nabla
    v_{\varepsilon}|^{q-2}}{v_{\varepsilon}^{q-3}}|D^2\ln{v_{\varepsilon}}|^2+\Gamma(q)\int_{\Omega}u_{\varepsilon}\frac{|\nabla
    v_{\varepsilon}|^q}{v_{\varepsilon}^{q-1}}\leq \frac{1}{\Gamma(q)}\int_{\Omega}v_{\varepsilon}(u_{\varepsilon}^{\frac{q+2}{2}}+1)
\end{align}
for all $t\in (0,T_{max,\varepsilon})$.

  Multiplying the first equation in \eqref{2.5} by $u_{\varepsilon}^{p-1}$, applying Young's inequality and \eqref{1.6}, there exists
  $c_1>0$ such that
    \begin{align}\label{3.32}
      &\frac{1}p \frac{d}{dt}\int_{\Omega}u_{\varepsilon}^{p}+\frac{p-1}{2}\int_{\Omega}u_{\varepsilon}^{p-1}v_{\varepsilon}|\nabla
      u_{\varepsilon}|^2\nonumber
      \\
      &\leq \frac{p-1}{2}\int_{\Omega}u_{\varepsilon}^{p+2\alpha-3}v_{\varepsilon}|\nabla
      v_{\epsilon}|^2+\ell\int_{\Omega}u_{\varepsilon}^pv_{\varepsilon}
      \\
      &\leq (p-1)\int_{\Omega}u_{\varepsilon}^{p+1}v_{\varepsilon}|\nabla v_{\varepsilon}|^2+(p-1)\int_{\Omega}v_{\varepsilon}|\nabla
      v_{\epsilon}|^2+\ell\int_{\Omega}u_{\varepsilon}^{p+1}v_{\varepsilon}+\ell\int_{\Omega}u_{\varepsilon}v_{\varepsilon}\nonumber
      \\
      &\leq  \frac{\Gamma(q)} 2
      \int_{\Omega}u_{\varepsilon}\frac{|\nabla v_{\varepsilon}|^q}{v_{\varepsilon}^{q-1}}+(p-1)\int_{\Omega}v_{\varepsilon}|\nabla
      v_{\varepsilon}|^2+\ell\int_{\Omega}u_{\varepsilon}^{p+1}v_{\varepsilon}+c_1\int_{\Omega}u_{\varepsilon}^{\frac{qp+q-2}{q-2}}v_{\varepsilon}^{\frac{3q-2}{q-2}}+\ell\int_{\Omega}u_{\varepsilon}v_{\varepsilon}\nonumber
      \\
      &\leq    \frac{\Gamma(q)} 2\int_{\Omega}u_{\varepsilon}\frac{|\nabla v_{\varepsilon}|^q}{v_{\varepsilon}^{q-1}}+(p-1)\int_{\Omega}v_{\varepsilon}|\nabla
      v_{\varepsilon}|^2+\ell\int_{\Omega}u_{\varepsilon}^{p+1}v_{\varepsilon}+
      c_1K^{\frac{2q}{q-2}}\int_{\Omega}u_{\varepsilon}^{\frac{qp+q-2}{q-2}}v_{\varepsilon}+
      \ell\int_{\Omega}u_{\varepsilon}v_{\varepsilon}\nonumber
    \end{align}
    for all $t\in (0,T_{max,\varepsilon})$.

     Noting that  if $p_0>\frac{3}{2}$, then for any  fixed $p>1$, one can find $q>2$ satisfying
\begin{align}\label{3.34}
    q\in \left( 2+\frac{3p}{p_0},\frac{4p_0}{3}+2p \right)
\end{align}
which ensures that
\begin{align}\label{3.35}
    \frac{qp+q-2}{q-2}<\frac{2p_0}{3}+p+1
\end{align}
and
\begin{align}\label{3.36}
  \frac{q+2}{2}<\frac{2p_0}{3}+p+1.
\end{align}
Therefore thanks to \eqref{3.15}, \eqref{3.35} and \eqref{3.36}, the application of Lemma \ref{lemma3.7} to $\eta:=\min\{\frac{\Gamma(q)} 4, \frac{p-1}4\}$  yields
\begin{align}\label{3.37}
   &\ell\int_{\Omega}u_{\varepsilon}^{p+1}v_{\varepsilon}+c_1K^{\frac{2q}{q-2}}\int_{\Omega}u_{\varepsilon}^{\frac{qp+q-2}{q-2}}v_{\varepsilon}+\frac{1}{\Gamma(q)}\int_{\Omega}u_{\varepsilon}^{\frac{q+2}{2}}v_{\varepsilon}\nonumber
   \\
   &\leq\eta\int_{\Omega}u_{\varepsilon}^{p-1}v_{\varepsilon}|\nabla u_{\varepsilon}|^2+\eta\int_{\Omega}u_{\varepsilon}\frac{|\nabla
   v_{\varepsilon}|^q}{v_{\varepsilon}^{q-1}}+C(\eta,K,p)\int_{\Omega}u_{\varepsilon}v_{\varepsilon}
\end{align}
for all $t\in (0,T_{max,\varepsilon})$.

From \eqref{3.31}, \eqref{3.32} and \eqref{3.37}, there exists $C(p, K)>0$ such that
\begin{align}\label{3.38}
   &\frac{d}{dt}\left\{ \frac{1}{p}\int_{\Omega}u_{\varepsilon}^{p}+\int_{\Omega}\frac{|\nabla
   v_{\varepsilon}|^q}{v_{\varepsilon}^{q-1}}\right\}
   +\frac{\Gamma(q)}{4}\int_{\Omega}u_{\varepsilon}\frac{|\nabla
   v_{\varepsilon}|^q}{v_{\varepsilon}^{q-1}}+\frac{p-1}{4}\int_{\Omega}u_{\varepsilon}^{p-1}v_{\varepsilon}|\nabla
   u_{\varepsilon}|^2\nonumber\\
   &\leq C(p,K)\int_{\Omega}\left( v_{\varepsilon}|\nabla v_{\varepsilon}|^2+u_{\varepsilon}v_{\varepsilon}+v_{\varepsilon} \right).
\end{align}
 Integrating \eqref{3.38} with respect to time, we  have
\begin{align*}
    &\int_{\Omega}u_{\varepsilon}^{p}+\int_{\Omega}\frac{|\nabla
    v_{\varepsilon}|^q}{v_{\varepsilon}^{q-1}}+\frac{\Gamma(q)}{4}\int_0^{t}\int_{\Omega}u_{\varepsilon}\frac{|\nabla
    v_{\varepsilon}|^q}{v_{\varepsilon}^{q-1}}+\frac{p-1}{4}\int_0^{t}\int_{\Omega}u_{\varepsilon}^{p-1}v_{\varepsilon}|\nabla
    u_{\varepsilon}|^2
    \\
    &\leq C(p,K)\int_0^{t}\int_{\Omega}\left(v_{\varepsilon}|\nabla v_{\varepsilon}|^2+u_{\varepsilon}v_{\varepsilon}+v_{\varepsilon}\right)
      +\int_{\Omega}(u_{0}+1)^p+\int_{\Omega}\frac{|\nabla v_{0}|^q}{v_0^{q-1}}
\end{align*}
for all $t\in (0,T_{max,\varepsilon})$. Hence  from \eqref{2.11}, \eqref{2.18}, \eqref{1.6} and \eqref{2.10}, it
follows that
\begin{align}\label{3.39}
    \int_{\Omega}u_{\varepsilon}^{p}+\frac{p-1}{2}\int_0^{t}\int_{\Omega}u_{\varepsilon}^{p-1}v_{\varepsilon}|\nabla u_{\varepsilon}|^2\leq C(p,
    K).
\end{align}
Furthermore, \eqref{3.30-1} results from \eqref{3.37}, \eqref{3.39} and \eqref{2.10}, and hence completes the proof.
\end{proof}

As a natural consequence of \eqref{3.30}, the application of standard smoothing estimates for the Neumann heat semigroup directly yields lower bounds for $v_{\varepsilon}$.
\begin{lemma}\label{lemma3.9}
Let $\Omega \subset \mathbb R^3$, $\alpha\in(\frac{3}{2},\frac{19}{12})$ and $K>0$ with the property that \eqref{1.6} holds. Assuming that $T_{max,\varepsilon}<\infty$,  then
there exists  $C=C( K,T_{max,\varepsilon})>0$ such that
\begin{align}\label{3.40}
   v_{\varepsilon}(x,t)\geq C\qquad for\; all\; (x,t)\in \Omega\times(0,T_{max,\varepsilon}).
\end{align}
\end{lemma}
\begin{proof}
   Firstly, it is  observed from Lemma \ref{lemma2.1} that $ v_\varepsilon \in C^{2,1}(\overline{\Omega} \times (0, T_{\max,\varepsilon})) $ with $v_\varepsilon>0$ in $\overline{\Omega} \times (0, T_{\max,\varepsilon})$. Defining the transformed variable $w_\varepsilon := \ln \frac{\|v_0\|_{L^\infty(\Omega)}}{v_\varepsilon} \in  C^{2,1}(\overline{\Omega} \times [0, T_{\max,\varepsilon}))$  and invoking \eqref{2.5}, we derive
$$w_{\varepsilon t} = \Delta w_\varepsilon - |\nabla w_\varepsilon|^2 + u_\varepsilon \leq \Delta w_\varepsilon + u_\varepsilon \quad \text{in } \Omega \times (0, T_{\max,\varepsilon}).$$
Using the comparison principle along with known regularization features of the Neumann heat semigroup $(e^{t\Delta})_{t \geq 0}$ on
 $ \Omega $ (\cite{Winkler4}),  one can see that there is
  constant $c_1> 0 $ such that  for all  $ t \in (0, T_{\max,\varepsilon})$,
\begin{align*}
w_\varepsilon (\cdot, t) &\leq e^{t\Delta}  w_{0\varepsilon} + \int_{0}^t e^{(t-s)\Delta} u_\varepsilon (\cdot, s) \, ds \\
&\leq \sup_{x \in \Omega} w_{0\varepsilon}(x) + c_1 \int_{0}^t (1+(t-s)^{-\frac{3}{4}} ) \|u_\varepsilon (\cdot, s)\|_{L^2(\Omega)} \, ds \\
&\leq \sup_{x \in \Omega}  w_{0\varepsilon}(x)  + c_1(T_{\max,\varepsilon}+ 4T^{\frac14}_{\max,\varepsilon})\cdot \sup_{s \in (0, T_{\max,\varepsilon})} \|u_\varepsilon (\cdot, s)\|_{L^2(\Omega)}.
\end{align*}
 This along with   \eqref{3.30} implies that  $w_\varepsilon (\cdot, t)\leq c_2(K,T_{\max,\varepsilon})$ for some $c_2(K,T_{\max,\varepsilon})$
  and thus $v_{\varepsilon}(x,t)\geq \|v_0\|_{L^\infty(\Omega)} e^{-c_2(K,T_{\max,\varepsilon})}$.
\end{proof}

\subsection{Proof of Theorem 1.1 }

On the basis of the result stated by Lemma \ref{lemma3.8}, we further derive the higher regularity properties of the solution components to \eqref{2.5} and then obtain a global continuous weak solution of \eqref{1.4}.

As a consequence of \eqref{3.30}, by leveraging standard heat semigroup estimates,
 one can verify that
 $ \|v_{\varepsilon}(\cdot,t)\|_{W^{1,\infty}(\Omega)}$ is locally bounded for $t\in (0, T_{max,\varepsilon})$.

\begin{lemma}\label{lemma3.10}
Let $\Omega \subset \mathbb R^3$ and $K>0$ with the property that \eqref{1.6} holds. Supposed that
$\alpha\in\left(\frac{3}{2},\frac{19}{12}\right)$, there exists  $C=C( K,p)>0$ such that for all $\varepsilon\in(0,1)$ we have
\begin{align}\label{3.40a}
   \|v_{\varepsilon}(\cdot,t)\|_{W^{1,\infty}(\Omega)}\leq C\qquad for\; all\; t\in (0,T_{max,\varepsilon}).
\end{align}
\end{lemma}
\begin{proof}
    According to Lemma \ref{lemma3.8}, we have the uniform boundedness of $u_{\varepsilon}$ in $L^{\infty}((0,T_{max,\varepsilon});L^p(\Omega))$ for
any $p>1$. Therefore, by invoking the well-established smoothing properties of the Neumann heat semigroup documented in \cite{Winkler4}, there exist $c_1 > 0$ and $c_2 > 0$ such that for any $p>3$
\begin{align}
\|\nabla v_\varepsilon(t)\|_{L^\infty(\Omega)}
&= \left\|\nabla e^{t(\Delta-1)}v_0 - \int_0^t \nabla e^{(t-s)(\Delta-1)} \big\{u_\varepsilon(s)v_\varepsilon(s) - v_\varepsilon(s)\big\} \, ds \right\|_{L^\infty(\Omega)} \nonumber\\
&\leq c_1 \|v_0\|_{W^{1,\infty}(\Omega)} + c_1 \int_0^t \left( 1 + (t-s)^{-\frac12 - \frac{3}{2p}} \right) e^{-(t-s)} \|u_\varepsilon(\cdot,s)v_\varepsilon(\cdot,s) - v_\varepsilon(\cdot,s)\|_{L^p(\Omega)} ds\nonumber\\
&\leq c_1 \|v_0\|_{W^{1,\infty}(\Omega)} + +c_1|\Omega|^{\frac1p} \|v_0\|_{L^\infty(\Omega)}\int_0^t \left( 1 + (t-s)^{-\frac12 - \frac{3}{2p}} \right) e^{-(t-s)} \, ds\nonumber\\
&+c_1 \|v_0\|_{L^\infty(\Omega)}\sup_{ t\in (0,T_{max,\varepsilon})} \|u_\varepsilon(\cdot,t)\|_{L^p(\Omega)} \int_0^t \left( 1 + (t-s)^{-\frac12 - \frac{3}{2p}} \right) e^{-(t-s)} \, ds\nonumber\\
&\leq c_2
\end{align}
for all $t \in (0, T_{\max,\varepsilon})$ and $\varepsilon \in (0,1)$. Then  we finish the proof.
\end{proof}

At this position, we can  proceed to prove  the $L^\infty$-bounds for $u_{\varepsilon}$.
\begin{lemma}\label{lemma3.11}
Let $\alpha\in\left(\frac{3}{2},\frac{19}{12}\right)$. Suppose $T_{max,\varepsilon}<\infty$, then there exists
$C( T_{max,\varepsilon})>0$
 such that for all $\varepsilon\in(0,1)$ we have
\begin{align}\label{3.41}
   \|u_{\varepsilon}(\cdot,t)\|_{L^{\infty}(\Omega)}\leq C( T_{max,\varepsilon}) ~\qquad for\; all\; ~t\in (0,T_{max,\varepsilon}).
\end{align}
\end{lemma}
\begin{proof}
We begin by reformulating the governing equation for $u_\varepsilon$ as
$$u_{\varepsilon t}=\nabla\cdot\left(A_{\varepsilon}(x,t,u_{\varepsilon})\nabla u_{\varepsilon}\right)+\nabla\cdot B_{\varepsilon}(x,t)+D_{\varepsilon}(x,t) \qquad(x,t)\in\Omega\times(0,T_{max,\varepsilon}),$$
where the nonlinear operators are defined by
$$A_{\varepsilon}(x,t,\xi):=v_{\varepsilon}(x,t)\xi \qquad(x,t,\xi)\in\Omega\times (0,T_{max,\varepsilon})\times[0,\infty),$$
and
$$B_{\varepsilon}(x,t):=-\chi u_{\varepsilon}^{\alpha}(x,t)v_{\varepsilon}(x,t)\nabla v_{\varepsilon}(x,t)\qquad(x,t)\in\Omega\times(0,T_{max,\varepsilon}),$$
as well as
$$ D_{\varepsilon}(x,t):=\ell u_{\varepsilon}(x,t)v_{\varepsilon}(x,t) \qquad(x,t)\in\Omega\times(0,T_{max,\varepsilon}).$$
Through applications of Lemma \ref{lemma3.8}, Lemma \ref{lemma3.9} and Lemma \ref{lemma3.10}, we derive the following estimates for any $p > 1$
$$A_{\varepsilon}(x,t)\geq c_{1}(T_{max,\varepsilon})\xi\qquad\text{for all }(x,t,\xi)\in\Omega\times (0,T_{max,\varepsilon})\times[0,\infty),$$
and
$$\|B_{\varepsilon}(\cdot,t)\|_{L^{p}(\Omega)}\leq c_{2}\quad\text{and}\quad\|D_{\varepsilon}(\cdot,t)\|_{L^{p}(\Omega)}\leq c_2\qquad\text{for all }t\in(0,T_{max,\varepsilon}),$$
with some $c_1>0,c_2>0$. Thereafter  \eqref{3.41}  can be derived from  a Moser-type iteration, as recorded in (\cite{TaoW}, Lemma A.1).
\end{proof}

As an application of \eqref{2.7} and Lemma \ref{lemma3.11}, one can show that the solutions to \eqref{2.5} are in fact global in time.
\begin{lemma}\label{lemma3.12}Under the assumptions of Theorem 1.1, we  have
$T_{max,\varepsilon}=\infty$ for all $\varepsilon\in(0,1).$
\end{lemma}

Drawing on the a priori estimates from the previous lemmas, we now deduce the
 H\"{o}lder regularity for the global solution to \eqref{2.5} by applying standard parabolic regularity theory.
\begin{lemma}\label{lemma3.13}
For all $T>0$ and
$\varepsilon\in(0,1)$, there exist $\theta_1=\theta_1(T)\in(0,1)$ and $C(T)>0$ such that
\begin{align}\label{3.42}
   \|u_{\varepsilon}\|_{C^{\theta_1,\frac{\theta_1}{2}}(\overline{\Omega}\times[0,T])} \leq C(T)
\end{align}
and
\begin{align}\label{3.43}
    \|v_{\varepsilon}\|_{C^{\theta_1,\frac{\theta_1}{2}}(\overline{\Omega}\times[0,T])} \leq C(T).
\end{align}
Moreover, for all $\tau>0$ and $T>\tau$ there exist $\theta_2=\theta_2(\tau,T)\in(0,1)$ and $C(\tau,T)>0$ such that
\begin{align}\label{3.44}
    \left\|v_{\varepsilon}\right\|_{C^{2+\theta_2,1+\frac{\theta_2}{2}}(\overline{\Omega}\times[\tau,T])}\leq C(\tau,T)\qquad\text{for all }\varepsilon\in(0,1).
\end{align}
\end{lemma}
\begin{proof}
   The H\"{o}lder estimates \eqref{3.42} and \eqref{3.43} are achieved through  the combination of Lemmas \ref{lemma3.9}--\ref{lemma3.11}, \eqref{2.8} and an application of the parabolic H\"{o}lder regularity framework established in \cite[Theorem 1.3; Remark 1.4]{Holder}. The subsequent estimate \eqref{3.44} then follows as a direct consequence of classical Schauder theory for scalar parabolic equations \cite[Chapter IV]{Schauder}, when coupled with the established H\"{o}lder continuity from \eqref{3.42}.
\end{proof}

Upon the above estimates, we can construct a continuous weak solutions to \eqref{1.4} through a standard extraction procedure.
\begin{lemma}\label{lemma3.14}
    Let $\alpha\in(\frac{3}{2},\frac{19}{12})$ and suppose that $(u_0,v_0)$ satisfies \eqref{1.5} and \eqref{1.6}. Then there exist
    $(\varepsilon_{j})_{j\in \mathbb{N}}\subset (0,1)$ and functions $u$ and $v$ fulfilling \eqref{1.8}, as well as $u\geq 0$ and $v>0$ in $\overline{\Omega}\times(0,\infty)$,
such that
\begin{align}
    &u_{\varepsilon}\rightarrow{u}\qquad \text{in}\quad C_{loc}^{0}(\overline{\Omega}\times[0,\infty))\label{3.45}\\
    &v_{\varepsilon}\rightarrow{v}\qquad \text{in}\quad C_{loc}^{0}(\overline{\Omega}\times[0,\infty))\quad \text{and}\quad\text{in}\quad C_{loc}^{2,1}(\overline{\Omega}\times(0,\infty))\qquad \text{and}\label{3.46}\\
    &\nabla v_{\varepsilon}\stackrel{*}{\rightharpoonup}\nabla v\qquad  \text{in}\quad L^{\infty}(\Omega\times(0,\infty))\label{3.47}
\end{align}
as $\varepsilon=\varepsilon_{j}\searrow0$, and that $(u,v)$ forms a continuous global weak solutions of \eqref{1.4} in the sense of
Definition \ref{Definition 2.1}.
\end{lemma}
\begin{proof}
    By virtue of Lemma \ref{lemma3.8}, Lemma \ref{lemma3.9} and Young's inequality, there exists $C=C(K,T)>0$ such that for all $T>0$,
    \begin{align*}
        \int_0^{T}\int_{\Omega}|\nabla u_{\varepsilon}^2|\leq \int_0^{T}\int_{\Omega}u_{\varepsilon}v_{\varepsilon}|\nabla u_{\varepsilon}|^2+\int_0^{T}\int_{\Omega}u_{\varepsilon}v_{\varepsilon}^{-1}\leq C,
    \end{align*}
    thereby establishing the boundedness of the gradient sequence $(\nabla u_{\varepsilon}^2)_{\varepsilon\in(0,1)}$ in $ L^1((0,T);W^{1,1}(\Omega))$. Subsequently, employing Lemma \ref{lemma3.10}, Lemma \ref{lemma3.13} and the Arzel\'{a}-Ascoli compactness theorem, a straightforward extraction procedure allows us to construct a vanishing subsequence $(\varepsilon_{j})_{j\in \mathbb{N}}$ with $\varepsilon_{j}\searrow0$. This yields nonnegative functions $(u,v)$ satisfying \eqref{1.8}, \eqref{2.2} and \eqref{3.45}--\eqref{3.47}. Finally, \eqref{2.3} and \eqref{2.4} can be accomplished in the respective weak  formulation associated with
      \eqref{2.5} on the basis of   these convergence properties.
\end{proof}
\begin{proof}[Proof of Theorem \ref{Th1.1}] The statements has fully been covered by Lemma \ref{lemma3.14}.
\end{proof}

\section{Large time behavior }\label{sec:4}
This section is devoted to establishing the stability properties and nontrivial stabilization results stated in Theorems \ref{Th1.2} and \ref{Th1.3}, respectively. The proofs rely crucially on the   temporal decay of $u_{\varepsilon t}$ in suitable dual spaces, which  is established via the integral estimates \eqref{2.10}, \eqref{2.11} and \eqref{2.19-1}.
\begin{lemma}\label{lemma4.1}
Let $\Omega \subset \mathbb R^3$ and $K>0$ with the property that \eqref{1.6} holds. Supposed that
$\alpha\in\left(\frac{3}{2},\frac{19}{12}\right)$, there exist  $\sigma>0$ and $C=C(K)>0$ such that for all $\varepsilon\in(0,1)$ we have
\begin{align}\label{4.1}
   \int_0^{\infty}\|u_{\varepsilon t}(\cdot,t)\|_{(W^{1,\infty}(\Omega))^*}dt\leq C\cdot \left\{\int_{\Omega} v_{0}\right\}^{\sigma}.
\end{align}
\end{lemma}
\begin{proof}
    Using the first equation in \eqref{2.5} one can see that for all $t>0$ and any $\psi\in W^{1,\infty}(\Omega) $ such that $\|\psi\|_{W^{1,\infty}(\Omega)}\equiv \max\{\|\psi\|_{L^{\infty}(\Omega)},\|\nabla \psi\|_{L^{\infty}(\Omega)}\}\leq 1$,
    \begin{align*}
   \left |\int_{\Omega}u_{\varepsilon t}\psi\right|&=\left|-\int_{\Omega}u_{\varepsilon}v_{\varepsilon}\nabla u_{\varepsilon}\cdot\nabla\psi +\chi \int_{\Omega}u_{\varepsilon}^{\alpha}v_{\varepsilon}\nabla v_{\varepsilon}\cdot\nabla\psi+\ell\int_{\Omega}u_{\varepsilon}v_{\varepsilon}\right|\\
    &\leq \int_{\Omega}u_{\varepsilon}v_{\varepsilon}|\nabla u_{\varepsilon}| +\chi \int_{\Omega}u_{\varepsilon}^{\alpha}v_{\varepsilon}|\nabla v_{\varepsilon}|+\ell\int_{\Omega}u_{\varepsilon}v_{\varepsilon},
    \end{align*}
 so that
 \begin{align}\label{4.2}
     \|u_{\varepsilon t}(\cdot,t)\|_{(W^{1,\infty}(\Omega))^*}\leq \int_{\Omega}u_{\varepsilon}v_{\varepsilon}|\nabla u_{\varepsilon}| +\chi \int_{\Omega}u_{\varepsilon}^{\alpha}v_{\varepsilon}|\nabla v_{\varepsilon}|+\ell\int_{\Omega}u_{\varepsilon}v_{\varepsilon}.
 \end{align}
 Thanks to the H\"{o}der inequality and \eqref{2.10}, for all $T>0$ we infer that
  \begin{align}\label{4.3}
     \int_0^{T}\int_{\Omega}u_{\varepsilon}v_{\varepsilon}|\nabla u_{\varepsilon}|
     &\leq\left\{\int_0^{T}\int_{\Omega}\frac{v_{\varepsilon}}{u_{\varepsilon}}|\nabla u_{\varepsilon}|^2\right\}^{\frac{1}{2}}\cdot \left\{\int_0^{T}\int_{\Omega}u_{\varepsilon}^3v_{\varepsilon} \right\}^{\frac{1}{2}}\nonumber\\
     &\leq\left\{\int_0^{T}\int_{\Omega}\frac{v_{\varepsilon}}{u_{\varepsilon}}|\nabla u_{\varepsilon}|^2\right\}^{\frac{1}{2}}\cdot \left\{\int_0^{T}\int_{\Omega}u_{\varepsilon}^{\frac{3-r}{1-r}}v_{\varepsilon} \right\}^{\frac{1-r}{2}}\cdot \left\{\int_0^{T}\int_{\Omega}u_{\varepsilon}v_{\varepsilon} \right\}^{\frac{r}{2}}\nonumber\\
     &\leq\left\{\int_0^{T}\int_{\Omega}\frac{v_{\varepsilon}}{u_{\varepsilon}}|\nabla u_{\varepsilon}|^2\right\}^{\frac{1}{2}}\cdot \left\{\int_0^{T}\int_{\Omega}u_{\varepsilon}^{\frac{3-r}{1-r}}v_{\varepsilon} \right\}^{\frac{1-r}{2}}\cdot \left\{\int_{\Omega}v_{0} \right\}^{\frac{r}{2}}
 \end{align}
 and
  \begin{align}\label{4.4}
     \int_0^{T}\int_{\Omega}u_{\varepsilon}^{\alpha}v_{\varepsilon}|\nabla v_{\varepsilon}|
     &\leq\left\{\int_0^{T}\int_{\Omega}v_{\varepsilon}|\nabla v_{\varepsilon}|^2\right\}^{\frac{1}{2}}\cdot \left\{\int_0^{T}\int_{\Omega}u_{\varepsilon}^{2\alpha}v_{\varepsilon} \right\}^{\frac{1}{2}}\nonumber\\
     &\leq\left\{\int_0^{T}\int_{\Omega}v_{\varepsilon}|\nabla v_{\varepsilon}|^2\right\}^{\frac{1}{2}}\cdot \left\{\int_0^{T}\int_{\Omega}u_{\varepsilon}^{\frac{2\alpha-r}{1-r}}v_{\varepsilon} \right\}^{\frac{1-r}{2}}\cdot \left\{\int_0^{T}\int_{\Omega}u_{\varepsilon}v_{\varepsilon} \right\}^{\frac{r}{2}}\nonumber\\
     &\leq\left\{\int_0^{T}\int_{\Omega}v_{\varepsilon}|\nabla v_{\varepsilon}|^2\right\}^{\frac{1}{2}}\cdot \left\{\int_0^{T}\int_{\Omega}u_{\varepsilon}^{\frac{2\alpha-r}{1-r}}v_{\varepsilon} \right\}^{\frac{1-r}{2}}\cdot \left\{\int_{\Omega}v_{0} \right\}^{\frac{r}{2}}
 \end{align}
 as well as
 \begin{align}\label{4.5}
    \int_0^{T}\int_{\Omega}u_{\varepsilon}v_{\varepsilon}\leq \int_{\Omega}v_{0}\leq\left\{\int_{\Omega}v_{0} \right\}^{\frac{r}{2}}\cdot\left(\|v_{0}\|_{L^{\infty}(\Omega)}\cdot|\Omega|\right)^{\frac{1-r}{2}},
 \end{align}
 where $r\in (0,1)$ and $\min\{\frac{2\alpha-r}{1-r},\frac{3-r}{1-r}\}>2$. It follows from \eqref{4.2}--\eqref{4.5} that there exists $C(K)>0$ such that
 \begin{align*}
      \int_0^T\|u_{\varepsilon t}(\cdot,t)\|_{(W^{1,\infty}(\Omega))^*}dt\leq C(K)\cdot\left\{\int_{\Omega}v_{0} \right\}^{\frac{r}{2}}
 \end{align*}
 due to \eqref{1.6}, \eqref{2.11}, \eqref{2.19-1} and \eqref{3.30-1}, and thereby we arrive at \eqref{4.1} with $\sigma:=\frac{r}{2}$.
\end{proof}
Making use of Lemma  \ref{lemma4.1}, we can derive the following consequence which  exclusively involves the zero-order expression of $u$. 
\begin{lemma}\label{lemma4.2}
Let $\alpha\in\left(\frac{3}{2},\frac{19}{12}\right)$ and $K>0$ with the property that \eqref{1.6} holds. Given
  $\sigma>0$ and $C=C(K)>0$ as in Lemma \ref{lemma4.1}, then for any nondecreasing $(t_{k})_{k\in \mathbb{N}}\subset [0,\infty)$, we have
\begin{align}\label{4.6}
   \sum_{k=1}^{\infty}\|u(\cdot,t_{k+1})-u(\cdot,t_{k})\|_{(W^{1,\infty}(\Omega))^*}dt\leq C\cdot \left\{\int_{\Omega} v_{0}\right\}^{\sigma},
\end{align}
where  we set $u(\cdot,0):=u_0$.
\end{lemma}
\begin{proof}
     Fixing any such nondecreasing sequence $(t_k)_{k \in \mathbb{N}}$, we infer from Lemma \ref{lemma4.1} that
     \begin{align*}
         \sum_{k\in \mathbb{N}} \| u_{\varepsilon}(\cdot, t_{k+1}) - u_{\varepsilon}(\cdot, t_k) \|_{(W^{1,\infty}(\Omega))^*}
    &= \sum_{k\in \mathbb{N}} \left\| \int_{t_k}^{t_{k+1}} u_{\varepsilon t}(\cdot, t) \, dt \right\|_{(W^{1,\infty}(\Omega))^*}\\
    &\leq \sum_{k\in \mathbb{N}} \int_{t_k}^{t_{k+1}} \| u_{\varepsilon}(\cdot, t) \|_{(W^{1,\infty}(\Omega))^*} \, dt\\
    &\leq C(K) \cdot \left\{ \int_{\Omega} v_{0} \right\}^{\sigma},
     \end{align*}
     because $(t_k, t_{k+1}) \cap (t_l, t_{l+1}) = \emptyset$ for $k \in \mathbb{N}$ and $l \in \mathbb{N}$ with $k \neq l$. This along with \eqref{3.45} implies \eqref{4.6} immediately.
\end{proof}
The quantitative dependence on $v_0$ not only establishes large-time stabilization of individual trajectories in their first component but also quantifies the proximity between the limiting profile and initial data.
\begin{lemma}\label{lemma4.3}
Let $\alpha\in\left(\frac{3}{2},\frac{19}{12}\right)$ and $K>0$ with the property that \eqref{1.6} holds.  Then  the function $u$ obtained in Lemma \ref{lemma4.2} exhibits the convergence
\begin{align}\label{4.7}
    u(\cdot,t)\rightarrow u_{\infty}\qquad in\;(W^{1,\infty}(\Omega))^*\qquad as\; t\to \infty
\end{align}
with some $u_{\infty}\in(W^{1,\infty}(\Omega))^*$ which satisfies
\begin{align}\label{4.7-1}
    \|u_{\infty}-u_0\|_{(W^{1,\infty}(\Omega))^*}\leq C(K)\cdot \left\{\int_{\Omega} v_{0}\right\}^{\sigma}
\end{align}
with $C(K)>0$  as in Lemma \ref{lemma4.1}.
\end{lemma}
\begin{proof}
    Lemma \ref{lemma4.2} implies that $\{u(\cdot,t_{k})\}_{k\in \mathbb{N}}$ forms a Cauchy sequence in $(W^{1,\infty}(\Omega))^*$, which establishes \eqref{4.7} with some $u_{\infty}\in(W^{1,\infty}(\Omega))^*$. Now select the sequence $(t_{k})_{k\in \mathbb{N}}$ with $t_1:=0$ and $t_k:=t$ for $k\geq 2$ in \eqref{4.6}, yielding
    \begin{align}\label{4.7-2}
        \|u(\cdot,t)-u_0\|_{(W^{1,\infty}(\Omega))^*}\leq C(K)\cdot \left\{\int_{\Omega} v_{0}\right\}^{\sigma}\qquad for \;all\; t>0,
\end{align}
and thereby derives \eqref{4.7-1} due to \eqref{4.7}.
\end{proof}

The quantitative form of the right-hand side in \eqref{4.7-1} and \eqref{4.7-2} allow us to derive the following stability property  of function pairs $(u_0,0)$.
\begin{lemma}\label{lemma4.4}
Let $\alpha\in\left(\frac{3}{2},\frac{19}{12}\right)$ and $K>0$ with the property that \eqref{1.6} holds.  Then for each $\eta>0$, there exists $\delta_1=\delta_1(K,\eta)>0$ whenever $u_0$ and $v_0$ fulfill
\eqref{1.5}, as well as
\begin{align}\label{4.11}
    \int_{\Omega}v_0\leq \delta_1,
\end{align}
the solution $(u,v)$ of \eqref{1.4} obtained in Theorem \ref{Th1.1} satisfies
\begin{align}\label{4.9}
    \|u(\cdot,t)-u_0\|_{(W^{1,\infty}(\Omega))^*}\leq \eta\qquad for \;all\; t>0.
\end{align}
Moreover, the corresponding limit function  from Lemma \ref{lemma4.3} admits
\begin{align}\label{4.10}
   \|u_{\infty}-u_0\|_{(W^{1,\infty}(\Omega))^*}\leq \eta.
\end{align}
\end{lemma}
\begin{proof}
   From \eqref{4.7-1} and \eqref{4.7-2}, it follows that for any $\eta>0$,  there exists $\delta_1=\delta_1(K,\eta)>0$ such that \eqref{4.11} warrants \eqref{4.9} and \eqref{4.10}.
\end{proof}

\begin{proof}[Proof of Theorem \ref{Th1.2}] The claimed result has precisely been asserted by Lemma \ref{lemma4.4}.
\end{proof}
The following lemma is used to obtain the decay estimate for $v$.
\begin{lemma}\label{lemma4.0}(\cite{TaoW3})
    Let $y\in C^1([0,\infty))$ and $h\in L^1_{loc}([0,\infty))$ both be nonnegative, and suppose that
    \begin{align*}
        \int_t^{t-1}h(s)ds\rightarrow 0\qquad as\; t\rightarrow\infty,
    \end{align*}
    and that there exists $\lambda>0$ such that
    \begin{align*}
        y'(t)+\lambda y(t)\leq h(t) \qquad for \;all\;t>0.
    \end{align*}
    Then
    \begin{align*}
        y(t)\rightarrow 0\qquad as\; t\rightarrow\infty.
    \end{align*}
\end{lemma}
The estimates established in \eqref{2.18} and \eqref{2.19-1}  imply the following asymptotic decay behavior for the second solution component.
\begin{lemma}\label{lemma4.5}
Let $\alpha\in\left(\frac{3}{2},\frac{19}{12}\right)$ and $K>0$ with the property that \eqref{1.6} holds.  Then  we have
\begin{align}\label{4.8}
    v(\cdot,t)\rightarrow 0\qquad in\;W^{1,p}(\Omega)\;for\;all\;p\geq 1\qquad as\; t\to \infty.
\end{align}
\end{lemma}
\begin{proof}
Thanks to \eqref{2.10},
\eqref{2.18}, \eqref{3.45} and
\eqref{3.46}, the application of  Fatou's lemma  then yields
$$\int_0^{\infty}\int_{\Omega}v<\infty \quad\text{and}\quad \int_0^{\infty}\int_{\Omega}uv<\infty.$$
Consequently,
\begin{align}\label{a}
    \int_{t-1}^{t}\int_{\Omega}v\rightarrow 0\quad\text{and}\quad\int_{t-1}^{t}\int_{\Omega}uv\rightarrow 0\qquad as\; t\to \infty.
\end{align}
Testing the second equation in \eqref{1.4} by $-\Delta v$, applying Young’s inequality together with H\"{o}lder's inequality, one can find $c_1>0$ such that
\begin{align}\label{e}
    \frac{d}{dt}\int_{\Omega}|\nabla v|^2&=-2\int_{\Omega}|\Delta v|^2+2\int_{\Omega}uv\Delta v\nonumber \\
    &\leq -\int_{\Omega}|\Delta v|^2+\int_{\Omega}u^2v^2\nonumber \\
    &\leq-\int_{\Omega}|\Delta v|^2+\|v_0\|_{L^{\infty}(\Omega)}\left(\int_{\Omega}u^3\right)^{\frac{1}{2}}\left(\int_{\Omega}uv\right)^{\frac{1}{2}}\nonumber \\
   &\leq -\int_{\Omega}|\Delta v|^2+c_1\left(\int_{\Omega}uv\right)^{\frac{1}{2}}
    \end{align}
for all $t>0$, due to \eqref{2.8} and \eqref{3.30}. Apart from that, integration by parts followed by Young’s inequality yields
\begin{align}\label{b}
    \int_{\Omega}|\nabla v|^2=\int_{\Omega}v\Delta v\leq \int_{\Omega}|\Delta v|^2+\frac{1}{4}\int_{\Omega}v^2.
\end{align}
 Combining \eqref{2.8}, \eqref{e} with \eqref{b}, we get
\begin{align}\label{c}
\frac{d}{dt}\int_{\Omega}|\nabla v|^2+\int_{\Omega}|\nabla v|^2\leq c_1\left(\int_{\Omega}uv\right)^{\frac{1}{2}}+\frac{\|v_0\|_{L^{\infty}(\Omega)}}{4}\int_{\Omega}v.
\end{align}
 Define $$y(t):=\int_{\Omega}|\nabla v|^2\qquad \text{and}\qquad g(t):=c_1\left(\int_{\Omega}uv\right)^{\frac{1}{2}}+\frac{\|v_0\|_{L^{\infty}(\Omega)}}{4}\int_{\Omega}v.$$ Then it follows from \eqref{c} that
\begin{align*}
    y'(t)+y(t)\leq g(t).
\end{align*}
Furthermore, by H\"{o}lder's inequality and \eqref{a}, we arrive at
\begin{align}\label{d}
\int_{t-1}^{t}g(t)\leq c_1\left(\int_{t-1}^{t}\int_{\Omega}uv\right)^{\frac{1}{2}}+\frac{\|v_0\|_{L^{\infty}(\Omega)}}{4}\int_{t-1}^{t}\int_{\Omega}v\rightarrow 0\qquad as\; t\to \infty.
\end{align}
As an application of Lemma \ref{lemma4.0}, we infer that
\begin{align}\label{4.8-1}
    \int_{\Omega}|\nabla v|^2\rightarrow 0\qquad as\; t\to \infty.
\end{align}
According to the interpolation inequality, \eqref{4.8-1} and \eqref{3.40a}, we obtain  for any $p>2$
\begin{align}\label{4.8-2}
    \|\nabla v(\cdot,t)\|_{L^{p}(\Omega)}\leq \|\nabla v(\cdot,t)\|_{L^{\infty}(\Omega)}^{\frac{p-2}{p}}\|\nabla v(\cdot,t)\|_{L^{2}(\Omega)}^{\frac{2}{p}}\leq C\|\nabla v(\cdot,t)\|_{L^{2}(\Omega)}^{\frac{2}{p}}\rightarrow 0 \quad \text{as} \quad t \to \infty.
\end{align}
Therefore \eqref{4.8} follows from \eqref{4.8-2}.
\end{proof}

  Beyond the stabilization  result established in Theorem 1.2 for solutions to
   \eqref{1.4}, we further demonstrate that  the  limiting  profile $u_{\infty}$ of the solution component $u$  obtained in \eqref{4.7} becomes non-homogeneous  when the initial signal concentration $v_0$ is sufficiently small, provided that $u_0$ is not identically constant.

\begin{lemma}\label{lemma4.6}
Let $\alpha \in \left( \frac{3}{2}, \frac{19}{12} \right)$ and $K > 0$ be such that \eqref{1.6} holds and suppose $u_0\not\equiv$ const. Then there exists $\delta_2 = \delta_2(K, u_0)>0$ such that if $u_0$ and $v_0$ satisfy \eqref{1.5} and \eqref{1.6}, as well as
\begin{align}\label{4.11-1}
    \int_{\Omega}v_0< \delta_2,
\end{align}
the corresponding limit function $u_{\infty} \in (W^{1,\infty}(\Omega))^*$ from Lemma \ref{lemma4.3} satisfies
$u_{\infty} \not\equiv \mathrm{const}$.
\end{lemma}
\begin{proof}
     Since $u_{0}$ is continuous and not constant, we can fix numbers $c_{1}>0$, $c_{2}>c_{1}$, as well as open sets $\Omega_1\subset\Omega$ and $\Omega_2\subset\Omega$, such that $u_{0}\leq c_{1}$ in $\Omega_1$ and $u_{0}\geq c_{2}$ in $\Omega_2$.  Furthermore,     let  functions $0\leq \psi_{i}\in C^{\infty}_{0}(\Omega)$, $i\in\{1,2\}$, such that $\operatorname{supp}\psi_{i}\subset \Omega_i$ and $\|\psi_{i}\|_{W^{1,\infty}(\Omega)}\leq1$.
     Choose $\delta_2= \delta_2(K,
     u_0)>0$ sufficiently small so that 
\begin{align}\label{4.12}
    C(K)\cdot\delta_2^{\sigma}<\kappa\cdot \min\left\{\int_{\Omega_1}\psi_1,\int_{\Omega_2}\psi_2\right\}
\end{align}
with $\kappa:=\frac{c_{2}-c_{1}}{4}$,  $C(K)>0$ and $\sigma>0$  given as in Lemma \ref{lemma4.1}.

Now if $u_{\infty}\not\in L^\infty(\Omega)$, then $u_{\infty} $ is non-constant   because any constant function is bounded, and the proof is thus complete.  Therefore, we may assume $u_{\infty} \in L^\infty(\Omega)$.  At this position, we claim  that
     \begin{align}\label{4.14}
        \esssup_{\Omega} u_{\infty}\geq c_{2}-\kappa\quad\text{and} \quad\essinf_{\Omega} u_{\infty}\leq c_{1}+\kappa.
    \end{align}
To verify this, supposed that 
     $\esssup_{\Omega} u_{\infty} < c_2 - \kappa$,  the properties of $\psi_2$ imply
\begin{align}\label{4.24}
    \|u_{\infty}-u_{0}\|_{(W^{1,\infty}(\Omega))^{*}}
    &\geq \left| \int_{\Omega} (u_{\infty} - u_0)\cdot \psi_2 \right| \nonumber\\
    &=\int_{\Omega_2}(u_{0}-u_{\infty})\cdot\psi_{2}\\
       &\geq\int_{\Omega_2}\{c_2+(\kappa-c_{2})\}\cdot\psi_{2} \nonumber\\
    &=\kappa\int_{\Omega_2}\psi_2,  \nonumber
\end{align}
On the other hand, Lemma \ref{lemma4.3} and \eqref{4.12} yield
\begin{align}\label{4.13}
    \|u_{\infty}-u_0\|_{(W^{1,\infty}(\Omega))^*} \leq  C(K)\cdot \left\{\int_{\Omega} v_{0}\right\}^{\sigma}
    <  \kappa \cdot\min\left \{\int_{\Omega_1}\psi_1,\int_{\Omega_2}\psi_2\right\},
\end{align}
which contradicts \eqref{4.24}.  The second inequality of \eqref{4.14}  follows analogously,  and thereby  $u_{\infty}$
 cannot coincide with any constant due to our choice of $\kappa$.
\end{proof}
\begin{proof}[Proof of Theorem \ref{Th1.3}] The non-triviality of the limiting profile $u_\infty$ in Theorem 1.3,
 along with the properties stated in \eqref{1.10-1},  follows directly from Lemma \ref{lemma4.3}--Lemma \ref{lemma4.6}, which completes the proof of the theorem.
\end{proof}

\vspace{.7cm}


\vspace{.3cm}

\noindent{\bf Conflict of interest}:
No potential conflict of interest is reported by the authors.

\noindent{\bf Ethics approval}:
Ethics approval is not required for this research.


\noindent{\bf Data availability statement}:
All data that support the findings of this study are included within the article. 

\noindent{\bf Financial support}: This work was supported by the National Natural Science Foundation of China (No. 12071030 and  No. 12271186) and  Beijing Key Laboratory on MCAACI.


\end{document}